\documentclass[11pt,oneside,openany]{article}
\usepackage{tikz}
\usepackage{tikz-cd}

\usepackage{amssymb}
\usepackage{amsthm}
\usepackage{amsmath,mathtools, amscd}
\usepackage{amsfonts}
\usepackage{latexsym}
\usepackage{graphics}
\usepackage{fancyhdr}
\usepackage{epstopdf}
\usepackage{setspace}
\usepackage{url}
\usepackage[top=1in, bottom=1in, left=.75in, right=.75in]{geometry}
\usepackage{hyperref}
\usepackage{qtree}
\usepackage{graphics}
\usepackage{graphicx}
\usepackage{times}
\usepackage{mathrsfs}
\usepackage[T1]{fontenc}
\usepackage{latexsym}
\usepackage{epsfig}
\usepackage{relsize}
\usepackage{bbm}
\usepackage{comment}
\usepackage{hyperref}
\usepackage{tikz}
\usepackage{tkz-fct}
\usepackage{qtree}
\usepackage{tikz-cd}
\usetikzlibrary{calc}
\usepackage{placeins}
\usepackage{color,soul}

\newcommand{\ans}[1]{
\fbox{\begin{minipage}{36em}
{\color{black}#1}\end{minipage}} \\
}

\newcommand{\e}{\epsilon}

\usepackage[vcentermath,enableskew]{youngtab}

\usepackage{enumerate}
\usepackage[shortlabels]{enumitem}
\setenumerate{label={\upshape(\roman*)},
  wide=0pt,topsep=0.2cm,labelsep=0.1em,itemsep=0.15cm,leftmargin=*}

\newtheorem{thm}{Theorem}[section]

\newtheorem{cor}[thm]{Corollary}
\newtheorem{lem}[thm]{Lemma}
\newtheorem{prop}[thm]{Proposition}

\theoremstyle{definition}

\newtheorem{defn}[thm]{Definition}

\def\R{{\mathbb R}} 
\def\A{{\mathbb A}} 
\def\Z{{\mathbb Z}} 

\title{Ergodicity and Algebraticity  of the Fast and Slow Triangle Maps}
\author{Thomas Garrity\footnote{Department of Mathematics and Statistics, Williams College, Williamstown, MA 01267, USA. Email: tgarrity@williams.edu} \and Jacob Lehmann Duke\footnote{Department of Mathematics, Dartmouth College, Hanover, NH 03755, USA. Email:jacob.lehmann.duke.gr@dartmouth.edu} }

\date{}
\begin{document}
\maketitle

\begin{abstract}
Our goal is to show that both the fast and slow versions of the triangle map (a type of multi-dimensional continued fraction algorithm)  in dimension $n$ are ergodic, resolving a conjecture of Messaoudi, Noguiera and Schweiger \cite{Mes} .  This particular type of  higher dimensional multi-dimensional continued fraction algorithm has recently been linked to the study of partition numbers, with the result that the underlying dynamics has combinatorial implications.

Keywords: \end{abstract}

\section{Introduction}\label{Introduction} The triangle map is a special type of multi-dimensional continued fraction algorithm.  Multidimensional continued fraction algorithms, of which there are many, are, in general, attempts to generalize the many wonderful properties of traditional continued fractions, such as for finding excellent rational approximations for a vector of real numbers and as for describing algebraic numbers via periodic sequences of integers (the Hermite Problem). (For more general information about multi-dimensional continued fractions, see Schweiger \cite{Sch} or Karpenkov \cite{Karpenkov 1}.)  Now, traditional continued fractions stem from the Gauss map, a map from the unit interval to itself.  Gauss realized that understanding the dynamics of this map is important.  Not surprisingly, once one has a multi-dimensional continued fraction algorithm, its associated dynamics is almost immediately studied. In 2009, 
Messaoudi, Nogueira, and Schweiger \cite{Mes} showed that the triangle map in dimension two is ergodic, and posed the problem of showing that higher dimensional analogs would also be ergodic.  Proving this conjecture is the main goal of this paper.  

In \cite{Mes}, the authors state in the abstract that ``As far as we know, it is the first example of a 2-dimensional algorithm where a surprising diophantine phenomenon happens: there are sequences of nested cells whose intersections is a segment, although no vertex is fixed.'' We strongly conjecture  the analog holds for the triangle map in general, meaning that for  the $n$-dimensional algorithm, there are sequences of nested cells whose intersections  are various simplices with possible dimension ranging from $1$ to $n-1$, still with  no vertex fixed.  

There is another motivating reason to study the dynamics of the triangle map. Recently in \cite{BBDGI,BG}, the triangle map has been shown to be linked with the classical combinatorial subject of partition numbers, unlike all other known multi-dimensional continued fraction algorithms.  Further, there is strong evidence that the underlying dynamics of the triangle map shapes the structure on the space of all partitions.  That was the prime motivation for us to try to show ergodicity in higher dimensions.

Our overall approach for showing that the triangle map is ergodic follows the outline of \cite{Mes}.  The key in the work of Messaoudi, Nogueira, and Schweiger is in showing that a certain nested sequence of cells converge to a single point, almost everywhere.  As standard in the rhetoric of these types of arguments, they show that for the $n=2$ triangle map, whenever infinite many of the terms in the triangle sequence are zero, then there is weak convergence.  Then they show that infinitely many of the terms of the triangle sequence are zero almost everywhere.  This is no longer a strong enough condition in general.  In Section \ref{weak convergence section} we show that to have  weak convergence  we need  infinitely many blocks of length $n-1$ in the triangle sequence to be bounded by a constant. Then in Section \ref{almost everywhere weak convergence}, we show that this happens almost everywhere.  Key are two combinatorial identities, shown in Subsection \ref{First Combinatorial Identity} and Subsection \ref{Second Combinatorial Identity}.  We view the need and the proof of these  two identities are what makes our proof of ergodicity not straightforward from the work in \cite{Mes}..

In Section \ref{basics triangle}, we review the basics of the triangle map. In Subsection \ref{Non-homgeneous Triangle Map}, we give the original definition of the triangle map (at least for dimension two), which is the version needed for partition theory.  In Subsection \ref{Homogeneous Triangle Map}, we give another version of the triangle map, switching the map from an iteration of an $n$ dimensional simplex to itself to a map that is the iteration of a cone over an $n$ dimensional simplex to itself.  This is a standard type of move in multi-dimensional continued fractions, allowing us to translate the maps to the language of $(n+1) \times (n+1)$ matrices.
In Subsection \ref{The sliced non-homogeous triangle map}, we look at the version of the triangle map used in \cite{Mes}.  This version is not directly useful in the study of partition numbers but is the version best suited for proofs of ergodicity. 

In Section \ref{weak convergence section}, we discuss what it means for an algorithm to weakly converge at a given point.  It is in this section that the behavior of the algorithm on the vertices of subcells is worked out.   Section \ref{almost everywhere weak convergence} contains the technical heart of the paper, and is what makes this paper not  merely  a straightforward generalization of \cite{Mes}.  Here we show that the triangle map converges weakly almost everywhere.  Key is a first combinatorial identity in Subsection \ref{First Combinatorial Identity} and a second combinatorial identity in  Subsection \ref{Second Combinatorial Identity}.

The goal of the first paper on the triangle map \cite{Gar} was in showing (in the dimension three case) that if the triangle sequence for a point $(\alpha_1,\alpha_2)$ is eventually periodic, then $\alpha_1$ and $\alpha_2$ are in the same number field of degree less than or equal to three.  With the work done in Section \ref{weak convergence section}, we show in Section \ref{algebratiticity} that if the triangle sequence for a point $(\alpha_1, \ldots, \alpha_n)$ is eventually periodic, then $\alpha_1, \ldots, \alpha_n$ are in the same number field of degree less than or equal to $n$.

In Section \ref{ergodic proof},  we use our earlier work to prove that the triangle map in any dimension is ergodic.  In Section \ref{fast invariant measure}, we find the invariant measure for the triangle map in each dimension.

Traditional continued fractions can be studied either via the Gauss map or by the Farey map.  The Gauss map is often also called the multiplicative version or the fast version of the map.  The Farey map can go by the name of the additive version or the slow version.  There is a natural interplay between the Gauss map and the Farey map.  The triangle map defined in Section \ref{basics triangle} is a direct generalization of the Gauss map.In Section \ref{slow version}, we look at the slow version (or the additive version or the Farey version, if you prefer) of the triangle map.  It is this slow version that is best suited to study partition numbers.  Hence this section, in which we prove that the slow version is also ergodic.

In Section \ref{generalization}, we quickly discuss why the generalization of the triangle map given in this paper is ``more correct'' than the version in \cite{Gar, Mes}.  Finally, in Section \ref{Conclusions}, we discuss further questions.

We would like to thank J. Fox, L. Pedersen and C. Silva for helpful conversations. 

\section{Background on the Triangle Map}\label{basics triangle}

The triangle map is a generalization of classical continued fractions. The $n=2$ case was   originally described in \cite{Gar, Ass}, where the concern was the number-theoretic  Hermite problem.    As mentioned in the introduction, Messaoudi,  Nogueira, and  Schweiger \cite{Mes} showed that the $n=2$  map is ergodic.  Our main goal is to generalize this to higher dimensions.   As mentioned in \cite{BBDGI}, ``further dynamical properties were discovered by  Berth\'e, Steiner and Thuswaldner \cite{Berthe-Steiner-Thuswaldner}  and by Fougeron and  Skripchenko \cite{Fougeron-Skripchenko}.   Bonanno, Del Vigna and  Munday \cite {Bonanno- Del Vigna-Munday} and Bonanno and Del Vigna \cite{Bonanno-Del Vigna} recently used the $\R^3$ slow  triangle map to develop a tree structure of rational pairs in the plane.   In a recent preprint, Ito \cite{Ito} showed that the fast  map is self-dual (in section three of that paper).  These papers are all primarily motivated by questions from dynamics.  For general background in multidimensional continued fraction algorithms, see Karpenkov \cite{Karpenkov 1} and Schweiger \cite{Sch}'' .

  \subsection{Non-homgeneous Triangle Map}\label{Non-homgeneous Triangle Map}
  Fix the simplex
  $$\triangle = \{ (x_1, \ldots, x_n) \in \R^n:1\geq x_1 \cdots \geq x_n\geq 0\}.$$
 Partition this simplex into sub-simplices
  
 \begin{eqnarray*}\triangle_b 
 &=&  \{ {\bf x}\in \triangle: 1-x_1 -bx_n \geq 0 >1-x_1- (b+1) x_n \} 
 \end{eqnarray*}
 Here 
 $$b=\left\lfloor \frac{1-x_1}{x_n} \right\rfloor$$
 
 \begin{defn}\label{triangle map} The triangle map
 $$T:\triangle \rightarrow \triangle$$ 
 is defined as 
$$T(x_1, \ldots, x_n) = \left(\frac{x_2}{x_1}, \ldots , \frac{x_n}{x_1}, \frac{1-x_1-b x_n}{x_1}\right)$$
for any $(x_1, \ldots , x_n) \in \triangle_b.$
 \end{defn}
  
 When $n=1$, this is the classical Gauss map.

 \begin{defn} The triangle sequence  for any  ${\bf x}\in \triangle$ is a sequence 
 $$(b_0,b_1, \ldots)$$
 of non-negative integers such that 
 $${\bf x }\in \triangle_{b_0}, T({\bf x}) \in \triangle_{b_1}, T^{(2)}({\bf x}) \in \triangle_{b_2}, \ldots .$$
 
 \end{defn}
 For any $x\in [0,1]$, its triangle sequence encodes  the classical continued fraction expansion of $x$.

 In turn, this suggests

 \begin{defn} For any sequence of non-negative integers   
 $$(b_0,b_1, \ldots),$$
 the corresponding cylinder is 
 $$\triangle (b_0, \ldots, b_k)= \{{\bf x } \in \triangle_{b_0}: T({\bf x}) \in \triangle_{b_1}, T^{(2)}({\bf x}) \in \triangle_{b_2}, \ldots, T^{(k)}({\bf x}) \in \triangle_{b_k}\}.$$

 \end{defn}

\subsection{Homogeneous Triangle Map}\label{Homogeneous Triangle Map}
Here we start with the cone
$$\triangle^H = \{(x_0, x_1, \ldots, x_n)\in \R^{n+1} : x_0 >x_1 > \cdots > x_n>0\}.$$
The superscript $H$ stands for ``homogeneous.''
We partition this cone into sub-cones
  
 \begin{eqnarray*}\triangle_b^H 
 &=&  \{ {\bf x}\in \triangle: x_0-x_1 -bx_n \geq 0 >x_0-x_1- (b+1) x_n \} 
 \end{eqnarray*}

Then 

 \begin{defn} The (homogeneous) triangle map
 $$T:\triangle^H \rightarrow \triangle^H$$ 
 is defined as 
$$T(x_0, \ldots, x_n) =(x_1, x_2, \ldots, x_n, x_0 - x_1 - bx_n)$$
for any $(x_0, \ldots , x_n) \in \triangle_b^H.$
 \end{defn}
 When we need to distinguish this map $T$ from the $T$ in the previous Subsection \ref{Non-homgeneous Triangle Map}, we will use $T^H$.
 
 If we write each ${\bf x}=(x_0, \ldots, x_n) $ as a column vector, we can describe the map $T$ via  matrix 
 multiplication:
 
$$ \begin{array} {ccccc}
 T({\bf x} ) &=& T_b\left(  \begin{array}{c} x_0 \\ \vdots \\ x_n \end{array}  \right) &\mbox{if} & {\bf x} \in \triangle_b^H\\
  &=& \left(  \begin{array}{cccccc}   0&1&0&\cdots &0 &0  \\
  0&0&1& \cdots &0&0 \\
   & & \vdots & & &  \\
   0&0&0& \cdots &0& 1 \\
   1 & -1 & 0 &\cdots &0& -b \end{array}  \right)   \left(  \begin{array}{c} x_0 \\ \vdots \\ x_n \end{array}  \right) &\mbox{if} & {\bf x} \in \triangle_b^H\\
 &=&    \left(  \begin{array}{c} x_1 \\ x_2 \\ \vdots \\ x_n \\ x_0-x_1-bx_n \end{array}  \right) & & 
 \end{array}$$

The link between the homogeneous and the non-homogeneous triangle maps is the following.  We have a projection
$$P:\triangle^H \rightarrow \triangle$$
given by 
$$P(x_0, \ldots , x_n) =\left( \frac{x_1}{x_0}, \ldots, \frac{x_n}{x_0} \right).$$
It is straightforward to show that $P:\triangle_b^H \rightarrow \triangle_b$ is onto.  What is key is that the following diagram is commutative:

$$\begin{CD}
\triangle^H @>T>>  \triangle^H \\
@VVPV   @VVPV \\ \triangle @ >T>> \triangle
\end{CD}$$

We have the embedding 
$$i:\triangle \rightarrow \triangle^H$$
given by
$$i(x_1, \ldots, x_n) =( 1, x_1, x_2, \ldots, x_n).$$
This allows us to translate calculations for the non-homogeneous triangle map into calculations for the homogeneous triangle map, which in turn allows us to use simple linear algebraic tools. 

Finally, we have 

 \begin{defn} The triangle sequence  for any  ${\bf x}\in \triangle^H$ is a sequence 
 $$(b_0,b_1, \ldots)$$
 of non-negative integers such that 
 $${\bf x }\in \triangle_{b_0}, T({\bf x}) \in \triangle_{b_1}, T^{(2)}({\bf x}) \in \triangle_{b_2}, \ldots .$$
 
 \end{defn}

 \begin{defn} For any sequence of non-negative integers   
 $$(b_0,b_1, \ldots),$$
 the corresponding cylinder is 
 $$\triangle^H (b_0, \ldots, b_k)= \{{\bf x } \in \triangle_{b_0}^H: T({\bf x}) \in \triangle_{b_1}^H, T^{(2)}({\bf x}) \in \triangle_{b_2}^H, \ldots, T^{(k)}({\bf x}) \in \triangle_{b_k}^H\}$$

 \end{defn}

\subsection{The sliced non-homogeous triangle map}\label{The sliced non-homogeous triangle map}
(While all of this is standard, our nomenclature of using the term``slice'' is not.)
For any vector 
${\bf x} = (x_0, x_1, \ldots, x_n)$, denote its $\ell_1$ norm as 
$$||{\bf x } || = |x_0| + \ldots + |x_n|.$$

Set 
$$\triangle^S = \{ { \bf x}\in \triangle^H : ||{\bf x}||=1\} .$$
(We are using the superscript ``S'' for slice, as $\triangle^S = \triangle^H \cap \{||{\bf x}|| = 1\}$.)
This is an n-dimensional simplex.  We have the natural projection
$$P^S:\triangle^H \rightarrow \triangle^S$$
given by 
$$P^S({\bf x}) = \frac{{\bf x}}{ ||{\bf x}||}.$$
We set 
$$\triangle_b^S= \{ {\bf x} \in \triangle^S \cap \triangle^H\}.$$

 \begin{defn} The sliced non-homogeous triangle map
 $$T:\triangle^S \rightarrow \triangle^S$$
 is
 $$T({\bf x})  =  \frac{T^H({\bf x})}{||T^H({\bf x})||},$$
 where the map $T^H$ on the right-hand-side is the homogeneous triangle map.
 \end{defn}
 When we need to distinguish this map $T$ from the $T$ in the previous Subsection \ref{Non-homgeneous Triangle Map} or from $T^H$, we will use $T^S$.

 \begin{defn}\label{triangle sequence} The triangle sequence  for any  ${\bf x}\in \triangle^S$ is a sequence 
 $$(b_0,b_1, \ldots)$$
 of non-negative integers such that 
 $${\bf x }\in \triangle_{b_0}^S, T({\bf x}) \in \triangle_{b_1}^S, T^{(2)}({\bf x}) \in \triangle_{b_2}^S, \ldots .$$
 
 \end{defn}

 The composition of 
 $$ P^S\circ i:\triangle \rightarrow \triangle^S$$
 gives us 
 \begin{eqnarray*}
     (x_1,x_2 \ldots) & \xrightarrow{i} & (1, x_1, x_2, \ldots)\\
     &\xrightarrow{P^S}& \left( \frac{1}{1+x_1+ \ldots + x_n},\frac{x_1}{1+x_1+ \ldots + x_n}, \ldots, \frac{x_n}{1+x_1+ \ldots + x_n}  \right)
 \end{eqnarray*}

 \begin{defn} For any sequence of non-negative integers   
 $$(b_0,b_1, \ldots),$$
 the corresponding cylinder is 
 $$\triangle^S (b_0, \ldots, b_k)= \{{\bf x } \in \triangle_{b_0}^S: T({\bf x}) \in \triangle_{b_1}^S, T^{(2)}({\bf x}) \in \triangle_{b_2}^S, \ldots, T^{(k)}({\bf x}) \in \triangle_{b_k}^S\}$$

 \end{defn}
 
We have 
 \begin{lem}\label{1-1} For all non-negative integers $b_0, b_1, b_2$, the map 
 $$\triangle(b_0, b_1, \ldots, b_k) \xrightarrow{P^S\circ i}\triangle(b_0, b_1, \ldots, b_k)  $$
 is a continuous, one-to-one, onto map.
 \end{lem}

\section{Weak Convergence} \label{weak convergence section}
This is the start of the key technical results for this paper.  From our theorem about when the higher dimensional triangle map does and does not exhibit weak convergence, we will be able to show, via fairly standard arguments, that the triangle map is ergodic and that eventual periodicity implies degree n irrationality. 

 \subsection{Definition of weak convergence}

 \begin{defn} Let ${\bf x} \in \triangle$ have triangle sequence $(b_0, b_1, \ldots).$ We say the the triangle map weakly converges to  ${\bf x}$ if 
 
 $$\lim_{m\rightarrow \infty} \cap_{k=0}^m \triangle(b_0, \ldots, b_m) = \{{\bf x}\}.$$

 \end{defn}
 
 By lemma \ref{1-1}, we have that the triangle map will weakly converge to ${\bf x} \in \triangle$  precisely when
 $$\lim_{m\rightarrow \infty} \cap_{k=0}^m \triangle^S(b_0, \ldots, b_m) = \{P^S\circ({\bf x})\}.$$
 We will be doing most of our work on convergence in the world of $\triangle^S$ and $\triangle^H.$
 
 In the $n=1$, and hence the classical continued fraction case, there is weak convergence for all ${\bf x}.$  As discussed in \cite{Gar, Ass, Mes}, this is not the case for $n=2$, when the limit $\lim{m\rightarrow \infty} \cap_{k=0}^m \triangle(b_0, \ldots, b_m) $ can be an entire line segment.  But the key technical part of the work of Messaoudi, Nogueira and Schweiger was their Theorem 1.1, where they show that weak convergence does happen if  infinitely many of the $b_k$ we have $b_k=0$.  Coupled with their Theorem 1.2, which states that  almost every point ${\bf x}\in \triangle$, with respect to Lebesgue measure, has infinitely many of their $b_k=0$,  allows them to prove ergodicity in the $n=2$ case.
 
 For higher dimensions,  the convergence of $\lim_{m\rightarrow \infty} \cap_{k=0}^m \triangle(b_0, \ldots, b_m) $ appears that it can be any simplex up to dimension $n-1$.  Further, just requiring infinitely many of the $b_k=0$ is no longer  a strong enough condition  to guarantee weak convergence. 
 
 One of our  main technical goals is 
 
\begin{thm}
\label{weakConvThm}
In the dimension $n$ case, let 
   $u \in \triangle^S$. If at infinitely many points in the triangle sequence for $u$, we have $n-1$ zeros in a row, then the triangle sequence is weakly convergent at $u$.
\end{thm}
This theorem will follow from our true 
 technical goal, which will take some work:
\begin{thm}
\label{}
Let $B$ be a constant. 
In the dimension $n$ case, 
    let $u \in \triangle^S$. If at infinitely many points in the triangle sequence for $u$, we have $n-1$ terms in a row being bounded by the constant $B$, then the triangle sequence is weakly convergent at $u$.
\end{thm}

In Corollary \ref{b_k=0 inf often cor}  we will show that  this happens almost everywhere with respect to Lebesgue measure, which in turn will allow us to finally prove ergodicity.

\subsection{Vertices of the subcells}\label{vertices}

This part is, relatively speaking, a straightforward generalization of section  4 of \cite{Ass} and of the introduction of \cite{ Mes}.
 
Fix a sequence $(b_0, b_1, \ldots)$ of non-negative integers.  Denote the vertices of $\triangle^H(b_0, \ldots, b_m)$ as
$$A_0(m), \ldots, A_n(m).$$
Then the vertices of $\triangle^S(b_0, \ldots, b_m)$ are 
$$ \frac{A_0(m)}{||A_0(m) ||}, \ldots,  \frac{A_n(m)}{||A_n(m) ||}.$$

The goal of Theorem \ref{weakConvThm}  is now equivalent to showing that   the distances between the vertices of $\triangle^S(b_0, \ldots, b_m)$ approach  zero:

\begin{lem} For a sequence $(b_0, b_1, \ldots)$ of non-negative integers, we have that 
$\lim_{m\rightarrow \infty} \cap_{k=0}^m \triangle^S(b_0, \ldots, b_m) $
is a single point if and only if, for all $i$ and $j$, 
$$\lim_{m\rightarrow \infty} \left|  \frac{A_i(m)}{||A_i(m) ||} -   \frac{A_j(m)}{||A_j(m) ||},\right|_2 =0.$$
\end{lem}

 Building on work in \cite{Ass, Mes}, we can recursively define the vertices of each $\triangle^H(b_0, \ldots, b_m)$.
 We have 
 \begin{lem}
 \[A_i(m+1)=A_{i+1}(m) \text{ for $i\in \{0, 2, \ldots, n-2\}$} \]
\[A_{n-1}(m+1)=b_{m+1}A_0(m)+A_n(m) \]
\[A_{n}(m+1)   = (b_{m+1}+1)A_0(m) +A_n(m).\]
 \end{lem}

For the dimension $n$ case,  the vertices of each $\triangle^H(b_0, \ldots, b_m)$ are vectors in $\R^{n+1}$.  The inital simplex $\triangle^H$ has vertices
$$A_0 = \left( \begin{array}{c} 1 \\ 0 \\ \vdots \\ 0 \end{array}  \right), A_1 = \left( \begin{array}{c} 1 \\ 1\\0 \\ \vdots \\ 0 \end{array}  \right), \ldots, A_n = \left( \begin{array}{c} 1 \\ 1 \\ \vdots \\ 1 \end{array}  \right)$$
giving us that the vertices of $\triangle^H(b_0)$ are 
$$A_0(0) = \left( \begin{array}{c} 1 \\ 1 \\ 0 \\ \vdots \\ 0 \end{array}  \right), A_1(0) = \left( \begin{array}{c} 1 \\ 1\\1\\ 0 \\ \vdots \\ 0 \end{array}  \right), \ldots, A_{n-2}(0) = \left( \begin{array}{c} 1 \\   \vdots \\ 1 \\ 0 \end{array}\right),$$
$$A_{n-1}(0) = \left( \begin{array}{c} b_0 \\ 1 \\ \vdots \\ 1 \end{array}  \right),   A_n(0) = \left( \begin{array}{c} b_0+1 \\ 1 \\ \vdots \\ 1 \end{array}  \right)$$
 
 As a schematic diagram, we can think of this as 
 
 
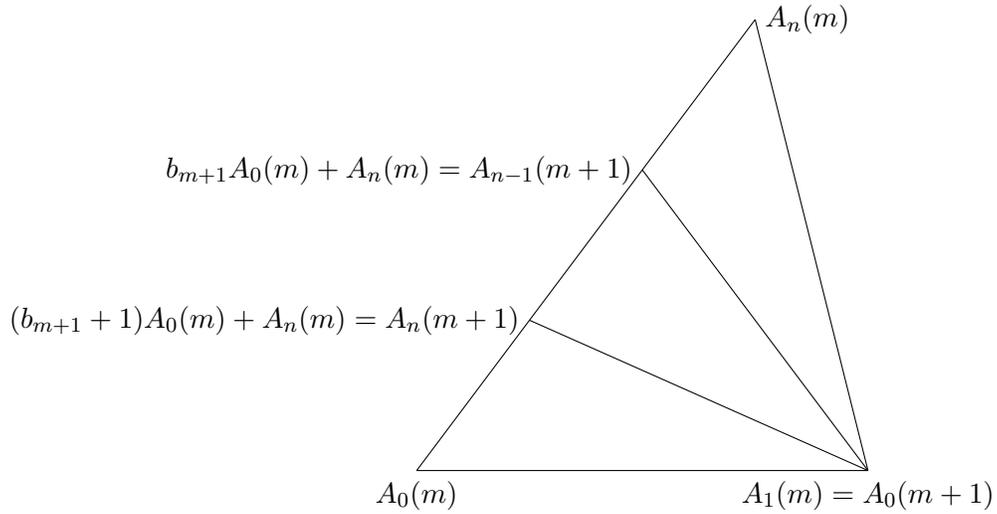
\begin{center}

 \begin{tikzpicture}[scale=1.5]

\draw(0,0)--(4,0);
\draw(4,0)--(3, 4);
\draw(3,4)--(0,0);

\draw(4,0)--(1,4/3);
\draw(4,0)--(2,8/3);

\node[below] at (0,0) {$A_0(m)$};
\node[below] at (4,0) {$A_1(m) = A_0(m+1)$};
\node[right] at (3,4) {$A_n(m)$};

\node[left] at (2, 8/3) {$b_{m+1}A_0(m)+ A_n(m) = A_{n-1}(m+1)$};
\node[left] at (1, 4/3) {$(b_{m+1}+1)A_0(m)+ A_n(m) = A_{n}(m+1)$};

\end{tikzpicture}  
      
\end{center}

Let $m>n.$  Suppose that there is a constant $B$ such that 
$$b_{m-(n-2)}, b_{m-(n-3)}, \ldots, b_m \leq B.$$
For notational convenience, set 
$$b_{m-(n-1)}=b,$$
where $b$ could be any nonnegative integer. Finally, again for notational convenience, denote the 
$$A_i(m-n) = A_i.$$

Then we have
\begin{lem}\label{recursive formula}
    \begin{eqnarray*}
A_0(m) &=& bA_0 + A_n   \\
A_1(m) &=& b_{m-(n-2)}A_1 + (b+1) A_0 + A_ n \\
A_2(m) &=& b_{m-(n-3)} A_2 + (b_{m-(n-2)}+1) A_1 + (b+1) A_0 + A_ n\\
A_3(m) &=&  b_{m-(n-4)} A_3 + (b_{m-(n-3)} +1) A_2 + (b_{m-(n-2)}+1) A_1 + (b+1) A_0 + A_ n \\
& \vdots & \\
A_{n-2}(m ) &=& b_{m-1} A_{n-2} + (b_{m-2}+1) A_{n-3} + \ldots +   (b_{m-(n-2)}+1) A_1 + (b+1) A_0 + A_ n   \\
A_{n-1}(m) &=&   b_mA_{n-1}  +(b_{m-1}+1) A_{n-2} + (b_{m-2}+1) A_{n-3} \\
&&
+ \ldots   + (b_{m-(n-2)}+1) A_1 + (b+1) A_0 + A_ n  \\
A_n(m ) &=&    (b_m+ 1) A_{n-1}  +(b_{m-1}+1) A_{n-2} + (b_{m-2}+1) A_{n-3} \\ 
&& 
+ \ldots   + (b_{m-(n-2)}+1) A_1 + (b+1) A_0 + A_ n     \\
\end{eqnarray*}

\end{lem}

\subsection{The n=3 example}

Just to ground the above in a concrete example, set
 $n=3$.  This means that each simplex 
$$\triangle^H(b_0, b_1,\ldots, b_m)$$
has $4$ vertices.

Suppose $m>4,$
set $$b_{m-2}= b,$$
and suppose that there is a constant $B$ so that 
$$b_{m-1}, b_m \leq B.$$
Denote the vertices of $\triangle^H(b_0, \ldots, b_{m-3})$  by 
$$A_0, A_1, A_2, A_3 .$$
Then in our notation we  have 
\begin{eqnarray*}
A_0(m-3) &=& A_0 \\
A_1(m-3) &=& A_1 \\
A_2(m-3) &=& A_2 \\
A_3(m-3) &=& A_3
\end{eqnarray*}

We have 
\begin{eqnarray*}
A_0(m-2) &=& A_1(m-3) \\
&=& A_1\\
A_1(m-2  )  &=& A_2(m-3) \\
&=& A_2 \\
A_{2}(m- 2) &=& b_{m-2}A_0 (m-3) + A_3(m-3)\\
&=& b  A_0   + A_3  \\
A_{3}(m- 2) &=& (b_{m-2}+1) A_0 (m-3) + A_3(m-3)\\
&=& (b +1) A_0   + A_3  \\
\end{eqnarray*}

Then
\begin{eqnarray*}
A_0(m-1) &=& A_1(m-2) \\
&=& A_2    \\
A_1(m-1  )  &=& A_2(m-2) \\
&=& b  A_0   + A_3 \\
A_{2}(m- 1) &=& b_{m-1}A_0 (m-2) + A_3(m-2)\\
&=& b_{m-1}  A_1 + (b +1) A_0   + A_3  \\
A_{3}(m- 1) &=& (b_{m-1}+1) A_0 (m-2) + A_3(m-2)\\
&=&  (b_{m-1} +1) A_1 + (b +1) A_0   + A_3    \\
\end{eqnarray*}

and

\begin{eqnarray*}
A_0(m) &=& A_1(m-1) \\
&=&  b  A_0   + A_3   \\
A_1(m  )  &=& A_2(m-1) \\
&=& b_{m-1}  A_1 + (b +1) A_0   + A_3  \\
A_{2}(m) &=& b_{m}A_0 (m-1) + A_3(m-1)\\
&=& b_{m}  A_2 + (b_{m-1} +1) A_1 + (b +1) A_0   + A_3   \\
A_{3}(m) &=& (b_{m}+1) A_0 (m-1) + A_3(m-1)\\
&=&  (b_{m} + 1 )  A_2 + (b_{m-1} +1) A_1 + (b +1) A_0   + A_3    \\
\end{eqnarray*}

This gives us the following:
\begin{eqnarray*}
    A_1 (m)  &=& b_{m-1}  A_1 +  A_0 + A_0(m)    \\
    A_2 (m)  &=&  b_{m}  A_2 + (b_{m-1} +1) A_1 +  A_0  + A_0(m)  \\
    A_3 (m)  &=&   (b_{m} + 1 )  A_2 + (b_{m-1} +1) A_1 +  A_0    + A_0(m)  \\
\end{eqnarray*}

  \subsection{Control of growth of vertices}

  We need to understand how the vertices $A_i(m)$ grow in their $\ell_1$ norm.
From the recursive formula and from that all of the $b_i$ are non-negative, we have
  \begin{lem}\label{vertices grow}
  For all $m$, we have 
  $$||A_0(m) || < || A_1(m)|| < \cdots < ||A_n(m)||$$
  and for all $i$,
  $$||A_i(m) || \leq ||A_i(m+1)||.$$
  \end{lem}
  
  But we also have that the ratios of the $$||A_k(m)||$$ cannot grow too quickly, under assumptions of control on the size of the various $b_k.$

  \begin{lem}\label{bounds on vertices} Let $m>n.$  Suppose that there is a constant $B$ such that 
$$b_{m-(n-2)}, b_{m-(n-3)}, \ldots, b_m \leq B.$$
Then for all $i<j$ we have 
$$\frac{ ||A_i(m)||}{||A_j(m)||}\geq \frac{1}{(B+1)(n+1)}$$
    
  \end{lem}

  \begin{proof}
  We have
\begin{eqnarray*} 
\frac{ ||A_i(m)||}{||A_j(m)||}&\geq & \frac{ ||A_0(m)||}{||A_n(m)||}\\
&=& \frac{ ||bA_0 + A_n||}{||(b_m+ 1) A_{n-1}  
+ \ldots   + (b_{m-(n-2)}+1) A_1 + (b+1) A_0 + A_ n || }\\
&&\mbox{using Lemma \ref{recursive formula} and its notation}\\
&\geq & \frac{ ||bA_0 + A_n||}{||(B+ 1) (bA_0 + A_n)  
+ \ldots   + (B+1) (bA_0 + A_n) + bA_0 + A_ n || }\\
&\geq & \frac{1}{(B+1)(n+1)}.
\end{eqnarray*}

  \end{proof}

\subsection{Weak convergence: proof of Theorem \ref{weakConvThm}}

We need a little more  notation. Given two non-zero vectors $U$ and $V$ in $\R^n_{\ge 0}$, define a distance function 
    \[D(U,V):=\left|\frac{U}{||U||}-\frac{V}{||V||} \right|_2.\]
We have 
    \begin{equation}\label{distance shortcut}
        D(U, U+V)=\left|\frac{U}{||U||}-\frac{V}{||U+V||} \right|_2=\frac{||V||}{||U+V||}D(U,V),
    \end{equation}    
which is straightforward, with a lower dimensional case originally proven in \cite{Mes}.

Set 
$$d(k) = \max_{i,j} D(A_i(k), A_j(k)).$$
Then $d(k)$ is the diameter of the simplex 
$$\triangle^S(b_0, \ldots, b_k).$$
We have 
$$d(0)\geq d(1) \geq d(2) \geq \ldots.$$
We will have weak convergence precisely when
$$\lim_{k\rightarrow \infty}d(k) = 0.$$

Let $m>n$   and assume  that there is a constant $B$ such that 
$$b_{m-(n-2)}, b_{m-(n-3)}, \ldots, b_m \leq B.$$
 We will show that 
 $$d(m) \leq \frac{Bn + B + n}{Bn + B + n + 1} d(m-n).$$

As there are infinitely many such $m$ with corresponding bounds $b_{m-(n-2)}, b_{m-(n-3)}, \ldots, b_m \leq B,$
we will have our limit.

One last point, before setting up the needed string of inequalities.  For all $j>i$, using Lemma \ref{recursive formula}, we know that 
both $A_i(m)$ and $A_j(m) - A_i(m)$ are in the simplex 
$\triangle^H(b_0, \ldots, b_{m-n})$, which gives us 
$$D(A_i(m), A_j(m) - A_i(m))\leq d(m-n).$$

For $i<j$, we have 
\begin{eqnarray*}
D(A_i(m), A_j(m) )&=& D(A_i(m), A_i+ (A_j(m)-A_i(m) ))\\
&=& \frac{||A_j(m)-A_i(m)||}{||A_j(m)||}D(A_i(m),A_j(m)-A_i(m))\\
&\leq & \frac{||A_j(m)-A_i(m)||}{||A_j(m)||}d(m-n)\\
&=& \left( 1 - \frac{||A_i(m)||}{||A_j(m)||}\right) d(m-n)\\
&\leq& \left( 1 -  \frac{1}{(B+1)(n+1)}\right) d(m-n)\\
&=& \frac{Bn + B + n}{Bn + B + n + 1} d(m-n).
\end{eqnarray*}

Theorem \ref{weakConvThm} is true.

\section{Bounds on triangle sequence almost everywhere }\label{almost everywhere weak convergence}

\subsection{Key Theorem on convergence almost everywhere}

This section is needed preliminaries for Section \ref{ergodic proof}, where we will finally show that the triangle map is ergodic.

Here we show

\begin{thm}
    \label{b_k<B inf often cor}
    Fix any non-negative number $B$.  For Lebesgue-almost-every $x\in \triangle$, the triangle sequence for $x$ contains an infinite number of strings of $b_i$  of length $n-1$ bounded by the constant $B$.
\end{thm}

We will actually prove a subset of the above also happens almost everywhere: 

\begin{thm}
    \label{b_k=0 inf often cor}
    For Lebesgue-almost-every $x\in \triangle$, the triangle sequence for $x$ contains an infinite number of strings of $n-1$ zeroes.
\end{thm}

This will give us 

\begin{cor}
\label{triangle seq weekly convergent cor}
    For Lebesgue-almost-every $u \in \Delta_{n-1}$, the triangle sequence is weakly convergent at $u$.
\end{cor}
\begin{proof}
    This follows immediately from the combination of Theorem \ref{weakConvThm} and Corollary \ref{b_k=0 inf often cor}. Note that even though with probability 1 a random point will be uniquely specified by its triangle sequence, we strongly conjecture that there are  examples of triangle sequences that converge to a line segment, a triangle, or a simplex of any dimension up to $n-2$.  The key point, though, is that the set of all points falling into one of these special examples has  Lebesgue measure zero.
\end{proof}

Proving Theorem \ref{b_k=0 inf often cor} will take work and is the most technical part of this paper.

\subsection{Key Proposition for weak convergence almost everywhere}
We will throughout assume that $m>n$, where $n$ is the dimension of the triangle map.
  \begin{prop}\label{prop on postive area}
Let $b_0, \ldots, b_{m-n}$ be any sequence of non-negative integers.  Let 
$$R = \cup_{b_i>0} \triangle^S(b_0, \ldots, b_{m-n},b,\underbrace{0,0,\ldots, 0}_{n-1 \text{ zeroes}}).$$

Then 
    \[\frac{\lambda(R)}{\lambda(\triangle^S(b_0, \ldots, b_{m-n})}>\frac{1}{(n-1)(n)(n+1)\cdots(2n-2)n^{n-1}}\]
    where $\lambda $ is Lebesgue measure.
\end{prop}

The proof of this proposition will take some work, spread out over the next few subsections.

Let us first, though,  see why this proposition will imply Theorem \ref{b_k=0 inf often cor}

\bigskip
{\it Proof of Theorem \ref{b_k=0 inf often cor} assuming Proposition \ref{prop on postive area} }

Every cell $ \triangle^S(b_0, \ldots, b_{m-n})$
has at least $\frac{1}{(n-1)(n)(n+1)(n+2)\cdots(2n-2)n^{n-1}}$ of its total area consisting of points such that of the next $n$ elements of the triangle sequence,  the last $n-1$ of them are sequential zeros. Since this is a positive percentage, the theorem follows. This is for the same reason as the fact that for any positive integer $s$, if we flip a biased coin an infinite number of times, we will get an infinite number of strings of tails of length $s$ with probability 1 no matter the bias of the coin (provided of course the probability of landing on tails is nonzero).

\subsection{Preliminary Lemma on Volume}

We will be using the notation as in Section \ref{vertices}.  
In particular, let $b_0, \ldots, b_{n-4}$ be any sequence of non-negative integers.  Let $A_0, \ldots, A_n$ be the column vectors in $\R^{n+1}$ that define the cone $\triangle^H(b_0, \ldots, b_{n-4})$, as described earlier.  Then 
$$\left\{\frac{A_0}{||A_0 || } , \ldots, \frac{A_n}{||A_n ||}\right\}$$
are the vertices of the simplex $\triangle^S(b_0, \ldots, b_{n-4})$.  We will need to use the following lemma (which is already known, though not quite in this language.

\begin{lem}
\label{cell volume lemma}
     We have  $$\lambda(\triangle^S(b_0, \ldots, b_{m-4}))=\frac{\sqrt{n}}{(n-1)!||A_1||\;||A_2||\cdots||A_n||}.$$ 
    Let $b\in \Z_{\ge 1}$.  Then
    $$\lambda(\triangle^S(b_0, \ldots, b_{m-4},b,\underbrace{0,0,\ldots, 0}_{n-2 \text{ zeroes}}))$$
    is 
    $$\frac{\sqrt{n}}{(n-1)!||bA_1+A_n||\;||(b+1)A_1+A_n||\;||(b+1)A_1+A_2+A_n||\cdots||(b+1)A_1+A_2+\ldots+A_n||}.$$
\end{lem}

\begin{proof}
    The first part of the lemma is proven in \cite{Mes} in their Lemma 3.2 on their page 292. We simply have to consider the $(n+1) \times (n+1)$ matrix whose columns are $A_0, \ldots, A_n,$ and observe that this matrix is the product of the matrices of the form $T_{b_k}$, as defined in Subsection \ref{Homogeneous Triangle Map}, each of which has determinant $\pm 1.$

    As seen in Lemma \ref{recursive formula}, the vertices of the cell $\triangle^S(b_0, \ldots, b_{m-4},b,\underbrace{0,0,\ldots, 0}_{n-2 \text{ zeroes}})$ are 
    $$A_1^n=bA_1+A_n, A_2^n=A_1^n+A_1, A_3^n=A_2^n+A_2, \ldots, A_n^n=A_{n-1}^n+A_{n-1}.$$ Thus plugging in for these vertices, we have the result.
\end{proof}

\subsection{First Combinatorial Identity}\label{First Combinatorial Identity}

Our  goal is to show that points with $n-1$ zeroes in a row early in their triangle sequence occupy an appreciable proportion of each cell. Our proof of this requires a pair of strictly combinatorial lemmas, which we present in this subsection and in the next. In addition to their relevance to the ergodicity result, they are interesting in their own right.

\begin{lem}
\label{alg lemma}
For any integers $0 \leq j \leq n,$ the following holds:   
$$\sum_{j=0}^{n-2} \frac{{n-2\choose j}(-1)^j}{(j+1)x+y}=\frac{(n-2)!\cdot x^{n-2}}{(x+y)(2x+y)(3x+y)\cdots((n-1)x+y)}.$$
\end{lem}

\ans{\textbf{The  $n=4$ example:} 
Before proving this strange lemma, we give an explicit example of what it is claiming. Plugging in $n=4$ on the left hand side, we get \begin{align*}
    &\frac{{2\choose0}(-1)^0}{x+y}+\frac{{2\choose1}(-1)^1}{2x+y}+\frac{{2\choose2}(-1)^2}{3x+y} \\
    =&\frac{1}{x+y}+\frac{-2}{2x+y}+\frac{1}{3x+y} \\
    =&\frac{2x+y-2(x+y)}{(x+y)(2x+y)}+\frac{1}{3x+y} \\
    =&\frac{-y}{(x+y)(2x+y)}+\frac{1}{3x+y} \\
    =&\frac{-y(3x+y)+(x+y)(2x+y)}{(x+y)(2x+y)(3x+y)} \\
    =&\frac{-3xy-y^2+2x^2+2xy+xy+y^2}{(x+y)(2x+y)(3x+y)} \\
    =&\frac{2x^2}{(x+y)(2x+y)(3x+y)},
\end{align*}
which matches the expression on the right hand side of the lemma.}

\begin{proof}
    We prove this by induction, with thanks to L. Pedersen, who suggested the proof strategy. First, factor out the $y$ from the denominators in the sum and perform the substitution $u=x/y$ to yield
    \begin{align}
        &\sum_{j=0}^{n-2} \frac{(-1)^j{{n-2}\choose j}}{(j+1)x+y} \notag \\ 
    =&  \frac{{{n-2}\choose {0}}}{x+y} - \frac{{{n-2}\choose {1}}}{2x+y}+\frac{{{n-2}\choose {2}}}{3x+y}-\ldots +(-1)^{n-2}\frac{{{n-2}\choose {n-2}}}{(n-1)x+y} \notag \\
    =& \frac{1}{y}\left[\frac{{{n-2}\choose {0}}}{1+u} - \frac{{{n-2}\choose {1}}}{1+2u}+\frac{{{n-2}\choose {2}}}{1+3u}-\ldots +(-1)^{n-2}\frac{{{n-2}\choose {n-2}}}{1+(n-1)u}\right] \label{third line algebra lemma}
    \end{align}  

    We now study the bracketed expression. Note that when $n=3$, the expression is $\frac{1}{1+u}-\frac{1}{1+2u}$, which equals $\frac{u}{(1+u)(1+2u)}$. For our inductive base case, we'll need something slightly stronger, which is that this expression holds for any pair of consecutive positive integer multiples of $u$, not just $u$ and $2u$. Let $r$ be the starting value. Our base case is \begin{equation}
        \frac{1}{1+ru}-\frac{1}{1+(r+1)u}=\frac{u}{(1+ru)(1+(r+1)u)}.
    \end{equation}

    Now suppose that for some positive integer $k\ge 3$, we have that for any positive integer $r$, \begin{align}
        \notag &\frac{{{k-2}\choose {0}}}{1+ru} - \frac{{{k-2}\choose {1}}}{1+(r+1)u}+\frac{{{k-2}\choose {2}}}{1+(r+2)u}-\ldots +(-1)^{k-2}\frac{{{k-2}\choose {k-2}}}{1+(r+k-2)u} \\ = & \frac{(k-2)!\cdot u^{k-2}}{(1+ru)(1+(r+1)u)\cdots(1+(r+k-2)u)} \label{inductive step}
    \end{align}

Recall Pascal's identity, namely ${n\choose {k-1}}+{n\choose{k}}={{n+1}\choose{k}}$. Applying this down the columns in the following equation subtraction, we have that 

\begin{table}[h]
\scalebox{0.8}{
\begin{tabular}{p{20cm}}
\begin{equation*} 
\frac{{{k-2}\choose {0}}}{1+ru} - \frac{{{k-2}\choose {1}}}{1+(r+1)u}+\frac{{{k-2}\choose {2}}}{1+(r+2)u}-\ldots +(-1)^{k-2}\frac{{{k-2}\choose {k-2}}}{1+(r+k-2)u}
\end{equation*}\\
\begin{equation*}
    \quad \quad \quad - \quad \quad \quad \left[\frac{{{k-2}\choose {0}}}{1+(r+1)u} - \frac{{{k-2}\choose {1}}}{1+(r+2)u}+\ldots +(-1)^{k-3}\frac{{{k-2} \choose {k-3}}}{1+(r+k-2)u}+(-1)^{k-2}\frac{{{k-2}\choose {k-2}}}{1+(r+k-1)u}\right]
\end{equation*} \\
\hline
\begin{equation*} 
    \quad \quad \frac{{{k-1}\choose {0}}}{1+ru} - \frac{{{k-1}\choose {1}}}{1+(r+1)u}+\frac{{{k-1}\choose {2}}}{1+(r+2)u}- \quad \ \ \ldots \ \  \quad +(-1)^{k-2}\frac{{{k-1}\choose {k-2}}}{1+(r+k-2)u}+(-1)^{k-1}\frac{{{k-2}\choose {k-2}}}{1+(r+k-1)u}   
\end{equation*}
\end{tabular}
}
\end{table}
\FloatBarrier

By \eqref{inductive step}, we know the first line in the subtraction equals $\frac{(k-2)!\cdot u^{k-2}}{(1+ru)(1+(r+1)u)\cdots(1+(r+k-2)u)}$ and the second equals $\frac{(k-2)!\cdot u^{k-2}}{(1+(r+1)u)(1+(r+2)u)\cdots(1+(r+k-1)u)}$. Subtracting these yields 
\begin{align}
    \notag &(k-2)!\cdot u^{k-2}\frac{(k-1)u}{(1+ru)(1+(r+1)u)\cdots(1+(r+k-1)u)} \\
    =&\frac{(k-1)!\cdot u^{k-1}}{(1+ru)(1+(r+1)u)\cdots(1+(r+k-1)u)} \label{induction conclusion}.
\end{align}

Thus if \eqref{inductive step} holds for $k$, it holds for $k+1$, so since we have the base case $k=3$ for all $r$, we know that \eqref{inductive step} holds for any $r$ and all $k \ge 3$.

Taking $r=1$ and $k=n$ in \eqref{inductive step}, we can now substitute back into \eqref{third line algebra lemma} to yield that our original summation equals \begin{align*}
    \frac{1}{y}\left[\frac{(n-2)!\cdot u^{n-2}}{(1+u)(1+2u)\cdots(1+(n-1)u)} \right]
\end{align*}

Now substituting back in for $u=x/y$, the $k-1$ factors of $1/y$ in the bracketed denominator cancel with the $k-2$ factors of $1/y$ in the numerator and the extra in front, leaving us with 
\[\sum_{j=0}^{n-2} \frac{{{n-2}\choose j}(-1)^j}{(j+1)x+y}=\frac{(n-2)!\cdot x^{n-2}}{(x+y)(2x+y)(3x+y)\cdots((n-1)x+y)}\] 
as desired.
      
\end{proof}

\subsection{ Second Combinatorial Identity}\label{Second Combinatorial Identity}

The second of the combinatorial lemmas is more difficult to state, so we introduce it with the help of a diagram. The relevant coefficients will arise naturally in our ergodicity proof, but we can more easily envision deriving them as follows.

Recall that the multinomial $n \choose{k_1, k_2, \ldots, k_m}$ is defined as the way, given $n$ items and $m$ boxes total, to choose $k_i$ items to put into each box $i$. We are interested particularly in the case where every item is placed into a box, meaning $\sum_{i=1}^m k_i=n$. Note that ${n \choose{k_1, k_2, \ldots, k_m}}=\frac{n!}{k_1!k_2!\cdots k_m!}$.

\begin{defn}
    A composition of a positive integer $n$ is a way of summing positive integers to yield $n$, where, unlike a partition, the order of the parts is taken into account. Thus $4=2+1+1$, $4=1+2+1$, and $4=1+1+2$ are all distinct compositions of $4$, although they all count as the same partition, namely $\{2,1\}\times[1,2]$. \cite{Andrews}
\end{defn}

There are $2^{n-1}$ compositions of $n$. We define their standard order to be lexicographic order with higher numbers coming first. Thus the standard order for compositions of 3 is $\{3\}, \{2,1\}, \{1,2\}, \{1,1,1\}$.

Given the composition $\{k_1, k_2, \ldots, k_m\}$, we naturally have a multinomial that corresponds, namely $${k_1+k_2+\ldots+k_m \choose k_1, k_2, \ldots, k_m}.$$
We can arrange all such multinomials into a binary tree as follows: let row $n$ contain all multinomials corresponding to compositions of $n$. Beginning with $1\choose 1$ at the top, every node $n\choose{k_1, k_2, \ldots, k_m}$ has left child

$${n+1\choose k_1+1, k_2, \ldots, k_m}$$
and right child
$${n+1\choose 1, k_1, k_2, \ldots, k_m }.$$

\begin{figure}
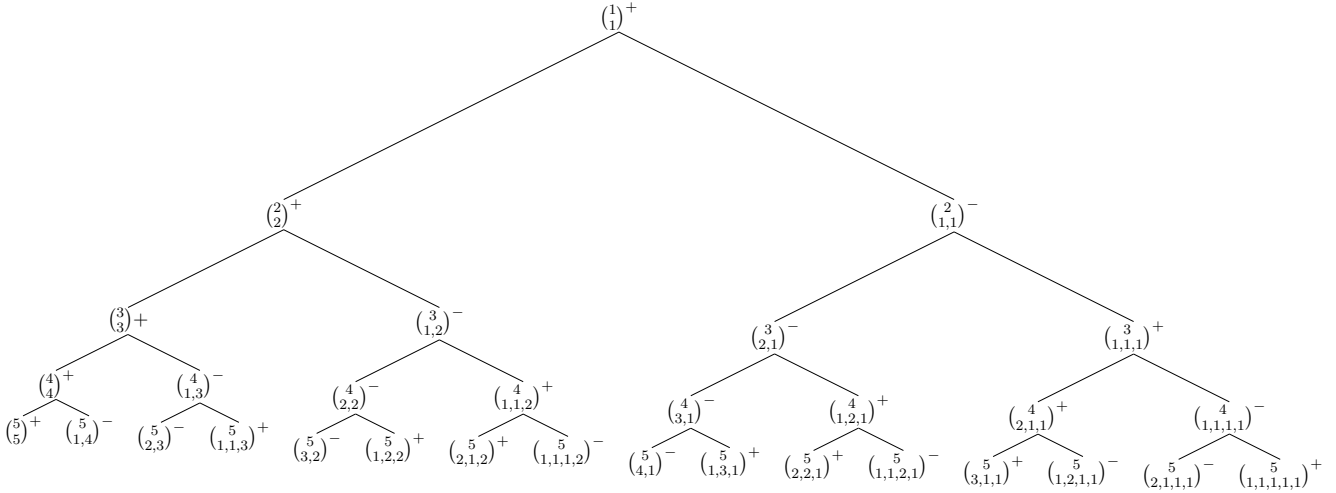

    \centering
    \scalebox{0.68}{
\Tree [.{$\binom{1}{1}^+$} 
        [.{$\binom{2}{2}^+$} 
            [.{$\binom{3}{3}+$} 
                [.{$\binom{4}{4}^+$} 
                    {$\binom{5}{5}^+$}
                    {$\binom{5}{1,4}^-$}
                ]
                [.{$\binom{4}{1,3}^-$} 
                    {$\binom{5}{2,3}^-$}
                    {$\binom{5}{1,1,3}^+$}
                ]
            ]
            [.{$\binom{3}{1,2}^-$} 
                [.{$\binom{4}{2,2}^-$} 
                    {$\binom{5}{3,2}^-$}
                    {$\binom{5}{1,2,2}^+$}
                ]
                [.{$\binom{4}{1,1,2}^+$} 
                    {$\binom{5}{2,1,2}^+$}
                    {$\binom{5}{1,1,1,2}^-$}
                ]
            ]
        ] 
        [.{$\binom{2}{1,1}^-$} 
            [.{$\binom{3}{2,1}^-$} 
                [.{$\binom{4}{3,1}^-$} 
                    {$\binom{5}{4,1}^-$}
                    {$\binom{5}{1,3,1}^+$}
                ]
                [.{$\binom{4}{1,2,1}^+$} 
                    {$\binom{5}{2,2,1}^+$}
                    {$\binom{5}{1,1,2,1}^-$}
                ]
            ]
            [.{$\binom{3}{1,1,1}^+$} 
                [.{$\binom{4}{2,1,1}^+$} 
                    {$\binom{5}{3,1,1}^+$}
                    {$\binom{5}{1,2,1,1}^-$}
                ]
                [.{$\binom{4}{1,1,1,1}^-$} 
                    {$\binom{5}{2,1,1,1}^-$}
                    {$\binom{5}{1,1,1,1,1}^+$}
                ]
            ]
        ] 
    ]
    }
    \caption{First Five Rows of Multinomial Tree}
    \label{Multinomial Tree Figure}
\end{figure}

We will now assign signs to the tree entries as follows: assign to $1\choose 1$ a $+$ and then as we move down the tree, sign each left child the same as its parent and each right child the opposite. Note that since left children have the same number of composition parts as their parents and right children have one more, this rule is the same as signing $n\choose{k_1, k_2, \ldots, k_m}$ by $(-1)^{m-1}$. See Figure \ref{Multinomial Tree Figure} for the first five rows of the tree with the signs corresponding to each entry labeled. Now the following strange property emerges, which we will first state in terms of the tree, then more formally with reference only to the standard order. For any $0\le k \le n$, the sum of the first $2^{k}$ signed entries in row $n+1$ equals $(-1)^{k}{n \choose k}$. 

As an example, take $k=3$ and $n=4$, which means we are looking at the sum of the first 8 signed entries in row 5. These are $1-5-10+20-10+30+30-60=-4$, which is ${4\choose 3}(-1)^3$, as predicted.

Let $C_{n,i}$ be the $i$-th composition in standard order of the positive integer $n$, and $|C_{n,i}|$ be the number of parts in $C_{n,i}$. Thus, for example, $C_{5, 7}=\{2,1,2\}$ and $|C_{5, 7}|=3$. Finally, let $n \choose C_{n,i}$ be the value of the multinomial corresponding to the composition $C_{n,i}$, so ${5 \choose C_{5,7}}={5\choose {2,1,2}}=\frac{5!}{2!1!2!}=30$. 

For clarity, Figure \ref{Multinomial Tree Figure New Notation} shows the same tree from Figure \ref{Multinomial Tree Figure}, just rewritten with the new notation.

\begin{figure}
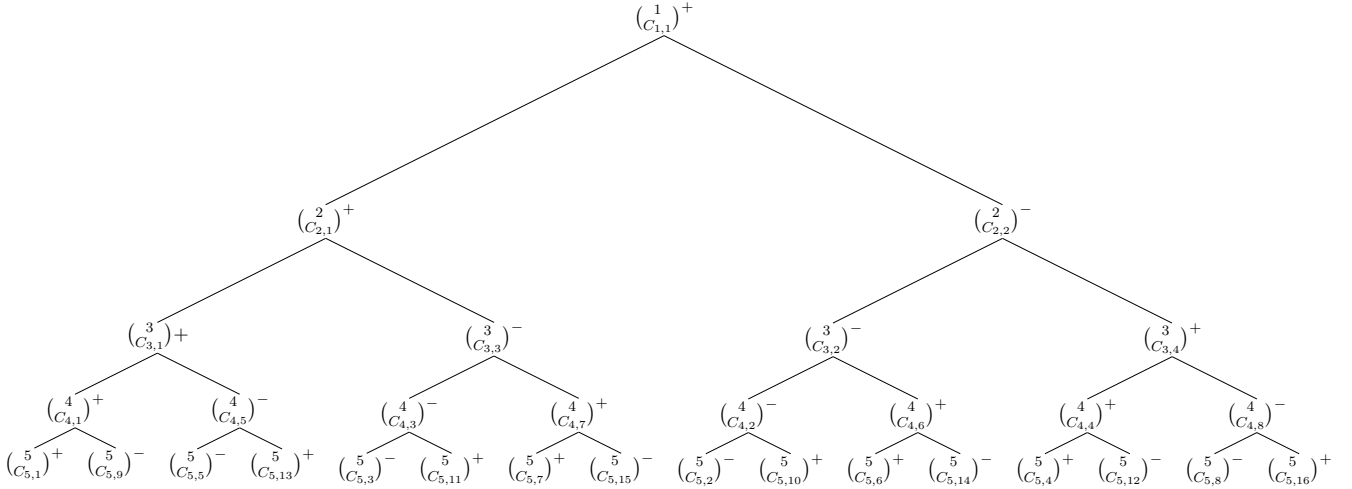

    \centering
    \scalebox{0.68}{
\Tree [.{$\binom{1}{C_{1,1}}^+$} 
        [.{$\binom{2}{C_{2,1}}^+$} 
            [.{$\binom{3}{C_{3,1}}+$} 
                [.{$\binom{4}{C_{4,1}}^+$} 
                    {$\binom{5}{C_{5,1}}^+$}
                    {$\binom{5}{C_{5,9}}^-$}
                ]
                [.{$\binom{4}{C_{4,5}}^-$} 
                    {$\binom{5}{C_{5,5}}^-$}
                    {$\binom{5}{C_{5,13}}^+$}
                ]
            ]
            [.{$\binom{3}{C_{3,3}}^-$} 
                [.{$\binom{4}{C_{4,3}}^-$} 
                    {$\binom{5}{C_{5,3}}^-$}
                    {$\binom{5}{C_{5,11}}^+$}
                ]
                [.{$\binom{4}{C_{4,7}}^+$} 
                    {$\binom{5}{C_{5,7}}^+$}
                    {$\binom{5}{C_{5,15}}^-$}
                ]
            ]
        ] 
        [.{$\binom{2}{C_{2,2}}^-$} 
            [.{$\binom{3}{C_{3,2}}^-$} 
                [.{$\binom{4}{C_{4,2}}^-$} 
                    {$\binom{5}{C_{5,2}}^-$}
                    {$\binom{5}{C_{5,10}}^+$}
                ]
                [.{$\binom{4}{C_{4,6}}^+$} 
                    {$\binom{5}{C_{5,6}}^+$}
                    {$\binom{5}{C_{5,14}}^-$}
                ]
            ]
            [.{$\binom{3}{C_{3,4}}^+$} 
                [.{$\binom{4}{C_{4,4}}^+$} 
                    {$\binom{5}{C_{5,4}}^+$}
                    {$\binom{5}{C_{5,12}}^-$}
                ]
                [.{$\binom{4}{C_{4,8}}^-$} 
                    {$\binom{5}{C_{5,8}}^-$}
                    {$\binom{5}{C_{5,16}}^+$}
                ]
            ]
        ] 
    ]
    }
    \caption{First Five Rows of Multinomial Tree with New Notation}
    \label{Multinomial Tree Figure New Notation}
\end{figure}

We can now state the lemma without reference to the tree.

\begin{lem}
\label{tree lemma}
For $n, k \in \Z_{\ge 0}$ with $k \le n$,
\[\underset{i \equiv 1\ (\text{mod } 2^{n-k})}{\mathlarger{\sum}} {n+1 \choose C_{n+1,i}}(-1)^{|C_{n+1,i}|-1}={n\choose k}(-1)^{k}.\]
\end{lem}

\begin{proof}
    We prove the lemma by induction, returning to the tree for intuition. Let $S_n$ be the sum of the multinomials corresponding to compositions of $n$ with their signs as listed in the tree. Equivalently, 
    \[S_n:=\sum_{i=1}^{2^{n-1}} {n \choose C_{n,i}}(-1)^{|C_{n,i}|-1}.\]

    As a base case, note that $S_1 = {1 \choose 1}=1$. Now suppose that for some $r \ge 0$, the lemma holds for row $r+1$ of the tree as well as all rows above it, meaning for any $0 \le k \le n \le r$, we have 
    \[\underset{i \equiv 1\ (\text{mod } 2^{n-k})}{\mathlarger{\sum}} {n+1 \choose C_{n+1,i}}(-1)^{|C_{n+1,i}|-1}={n\choose k}(-1)^{k}.\]

    The key thing to note is that every time we move right down the tree, we fix the last composition part size for all descendants of that node. Thus the entire right half of the tree consists of multinomials corresponding to compositions ending in 1, the second quarter of the tree from the left consists of multinomials corresponding to compositions ending in 2, and in general the second $\frac{1}{2^j}$-th of the tree from the left consists of multinomials corresponding to compositions ending in $j$. Thus in row $r+1$, where we have $2^{r}$ entries in the tree, we see that reading them from left to right, we encounter first 1 multinomial corresponding to a composition ending in $r+1$ (the odd one out in a sense, since it isn't the second unit in any binary division), then 1 multinomial corresponding to a composition ending in $r$, then 2 multinomials corresponding to a composition ending in $r-1$, and so on, with $2^{j}$ compositions corresponding to multinomials ending $r-j$. These compositions correspond to the ways to place $j+1$ items into boxes, since we have $r+1$ items total but are committing to having $r-j$ of them in our last box. Since there are only $2^j$ compositions of $j+1$, their corresponding multinomials must all appear in row $r+1$, each with a final part of size $r+1-(j+1)=r-j$ tacked on. In fact, our tree rules are such that those multinomials appear in the same order in row $r+1$ that they did originally in row $j+1$, just each with the extra part added to make the components sum to $r+1$. Thus reading the compositions from each multinomial across the tree at row $r+1$, we get first the entry ${r+1}$, then what is essentially a copy of row 1, just with an $r$ tacked on, then a copy of row 2 with an $r-1$ tacked on to each entry, then a copy of row 3 with an $r-2$ tacked on to each entry, and so on. 
    
    Since we are adding an entry to each composition, all the signs swap from how the appear in row $j+1$ to how they appear in row $r+1$. The effects of replacing the sum from row $j+1$ with the copy of row $j+1$ that appears in row $r+1$ are as follows: first, the multinomial now has us picking from $r+1$ items instead of $j+1$ items, so the $(j+1)!$ in the numerator of the value in row $j+1$ is replaced by an $(r+1)!$, for a net effect of multiplying by $\frac{(r+1)!}{(j+1)!}$. Second, the tacked on $r-j$ contributes an $(r-j)!$ to the denominator of the version of the multinomial that appears in row $r+1$. And third, as noted, the signs all swap from row $j+1$ to the copy of row $j+1$ embedded in row $r+1$. Thus the overall effect is that reading from left to right along row $r+1$, we have that the entries from number $2^{j-1}+1$ to number $2^{j}$ sum to $-S_j\frac{(r+1)!}{(j+1)!}\frac{1}{(r-j)!}$, or equivalently, to $-S_j\binom{r+1}{j+1}$. Since $j < r+1$, we have by our inductive assumption that $S_j=(-1)^{j-1}$. We now have that the sums of chunks of increasing size reading along row $r+1$ are as follows:
    \begin{itemize}
        \item The first entry, which is $\binom{r+1}{r+1}=1$
        \item The second entry, which is $-S_1\binom{r+1}{1}$
        \item The third and fourth entries, which sum to $-S_2\binom{r+1}{2}$
        \item Entries 5-8, which sum to $-S_3\binom{r+1}{3}$
        \item And so on, with entries $2^{j-1}+1$ to $2^{j}$ summing to $-S_j\binom{r+1}{j}$
    \end{itemize}

    Now we plug in for $S_n=(-1)^{n-1}$ and add up the items above. Writing the initial 1 as $\binom{r}{0}$, we have that by Pascal's identity, the sum of the first two entries in row $r+1$ is $\binom{r}{0}-\binom{r+1}{1}=-\binom{r}{1}$. Then the sum of the first four entries is $-\binom{r}{1}+\binom{r+1}{2}=\binom{r}{2}$. Continuing, we see that the sum of the first $2^j$ entries in row $r+1$ is $\binom{r}{j}(-1)^j$. Taking $j=r$, we get that the sum of the entire row, $S_{r+1}$, is equal to $\binom{r}{r}(-1)^{r}=(-1)^r$, allowing us to use this in the simplification for future rows, as desired. This completes the proof of the lemma.
\end{proof}

\subsection{A lower bound }

To prove the key Proposition \ref{prop on postive area} for weak convergence, we will need 
\begin{lem}
\label{lin alg craziness lemma} Using the notation of Proposition \ref{prop on postive area}, we have 
     $$\sum_{b=1}^\infty \lambda(\cup_{b>0} \triangle^S(b_0, \ldots, b_{m-n},b,\underbrace{0,0,\ldots, 0}_{n-1 \text{ zeroes}}))$$
     is strictly greater than 
   $$ \frac{\sqrt n}{(n-1)!(n-1)||A_0||}\left(\frac{1}{||A_0+A||\;||2A_0+A||\cdots||(n-1)A_0+A||}\right),$$
     where $A:=A_1+A_2+\ldots+A_n.$
\end{lem}

Before we prove this lemma, we give an example for the $n=5$ case that illustrates the strategy. Then we have
     \begin{align*}
        &\frac{1}{||bA_0+A_5||\;||(b+1)A_0+A_5||\;||(b+1)A_0+A_1+A_5||\cdots||(b+1)A_0+A_1+\ldots+A_5||} \\
        >& \frac{1}{||bA_0+A||\;||(b+1)A_0+A||\;||(b+2)A_0+A||\cdots||(b+4)A_0+A||} \\
        =& \frac{1}{||A_0||}\left(\frac{1}{||bA_0+A||}-\frac{1}{||(b+1)A_0+A||}\right)\frac{1}{||(b+2)A_0+A||\;||(b+3)A_0+A||\;||(b+4)A_0+A||} \\
    \end{align*} 
    \begin{align*}   
    \begin{split}
        =\frac{1}{||A_0||}\left[\frac{1}{2||A_0||}\left(\frac{1}{||bA_0+A||}-\frac{1}{||(b+2)A_0+A||}\right) -\frac{1}{||A_0||} \left(\frac{1}{||(b+1)A_0+A||} \right.\right. \\ \left.\left.-\frac{1}{||(b+2)A_0+A||}\right)\right] \cdot \frac{1}{||(b+3)A_0+A||\;||(b+4)A_0+A||} 
    \end{split}
    \end{align*}
    \begin{align*} 
    \begin{split}
        =\frac{1}{||A_0||}\left[\frac{1}{2||A_0||}\left(\frac{1}{3||A_0||}\left(\frac{1}{||bA_0+A||}-\frac{1}{||(b+3)A_0+A||}\right)-\frac{1}{||A_0||}\left(\frac{1}{||(b+2)A_0+A||} \right.\right.\right. \\
         \left.\left.\left. -\frac{1}{||(b+3)A_0+A||}\right)\right)-\frac{1}{||A_0||}\left(\frac{1}{2||A_0||}\left(\frac{1}{||(b+1)A_0+A||}-\frac{1}{||(b+3)A_0+A||}\right) \right.\right. \\
         \left.\left. -\frac{1}{||A_0||}\left(\frac{1}{||(b+2)A_0+A||}-\frac{1}{||(b+3)A_0+A||}\right)\right)\right]\cdot \frac{1}{||(b+4)A_0+A||}
    \end{split}
    \end{align*}

    \begin{align*}
    \begin{split}
        = \frac{1}{||A_0||}\left[\frac{1}{2||A_0||}\left(\frac{1}{3||A_0||}\left[\frac{1}{4||A_0||}\left(\frac{1}{||bA_0+A||}-\ast \right)-\frac{1}{||A_0||}\left(\frac{1}{||(b+3)A_0+A||}-\ast \right)\right] \right.\right.\\ \left.\left. -\frac{1}{||A_0||}\left[\frac{1}{2||A_0||}\left(\frac{1}{||(b+2)A_0+A||}-\ast\right)-\frac{1}{||A_0||}\left(\frac{1}{||(b+3)A_0+A||}-\ast \right)\right]\right) \right. \\ \left. -\frac{1}{||A_0||}\left(\frac{1}{2||A_0||}\left[\frac{1}{3||A_0||}\left(\frac{1}{||(b+1)A_0+A||}-\ast\right)-\frac{1}{||A_0||}\left(\frac{1}{||(b+3)A_0+A||}-\ast\right)\right] \right.\right. \\  \left.\left. -\frac{1}{||A_0||}\left[\frac{1}{2||A_0||}\left(\frac{1}{||(b+2)A_0+A||}-\ast\right)-\frac{1}{||A_0||}\left(\frac{1}{||(b+3)A_0+A||}-\ast\right)\right]\right)\right]
    \end{split}
    \end{align*}
    \begin{align*}
    & \text{where } \ast:=\frac{1}{||(b+4)A_0+A||}
    \end{align*}

Thus when we take the sum over all $b$ from $1$ to $\infty$, the tails cancel in each difference, leaving us with only a few leading terms in each set of parentheses. Factoring the $\frac{1}{||A_0||^4}$ common to all the terms, we have that the sum equals 

\begin{align*}
\begin{split}
\frac{1}{||A_0||^4}\Bigg[\frac{1}{24}\left(\frac{1}{||A_0+A||}+\frac{1}{||2A_0+A||}+\frac{1}{||3A_0+A||}+\frac{1}{||4A_0+A||}\right)-\frac{1}{6} \left(\frac{1}{||4A_0+A||}\right) \\ -\frac{1}{4}\left(\frac{1}{||3A_0+A||}+\frac{1}{||4A_0+A||}\right)+\frac{1}{2}\left(\frac{1}{||4A_0+A||}\right) \\ -\frac{1}{6}\left(\frac{1}{||2A_0+A||}+ \frac{1}{||3A_0+A||}+ \frac{1}{||4A_0+A||}\right) +\frac{1}{2}\frac{1}{||4A_0+A||} \\ +\frac{1}{2}\left(\frac{1}{||3A_0+A||}+\frac{1}{||4A_0+A||}\right)-1\left(\frac{1}{||4A_0+A||}\right)\Bigg]
\end{split}
\end{align*}
\begin{align*}
\begin{split}
=\frac{1}{||A_0||^4}\Bigg[\frac{1}{||A_0+A||}\left(\frac{1}{24}\right) +\frac{1}{||2A_0+A||}\left(\frac{1}{24}-\frac{1}{6}\right)+\frac{1}{||3A_0+A||}\left(\frac{1}{24}-\frac{1}{4}-\frac{1}{6}+\frac{1}{2}\right) \\ +\frac{1}{||4A_0+A||}\left(\frac{1}{24}-\frac{1}{6}-\frac{1}{4}+\frac{1}{2}-\frac{1}{6}+\frac{1}{2}+\frac{1}{2}-1\right)\Bigg] 
\end{split}
\end{align*}

\begin{align*}
\begin{split}
=\frac{1}{24||A_0||^4}\Bigg[\frac{1}{||A_0+A||}\left(1\right)+\frac{1}{||2A_0+A||}\left(1-4\right)+\frac{1}{||3A_0+A||}\left(1-6-4+12\right)  \\  +\frac{1}{||4A_0+A||}\left(1-4-6+12-4+12+12-24\right)\Bigg]
\end{split}
\end{align*}
\begin{align*}
\begin{split}
=\frac{1}{24||A_0||^4}\left(\frac{1}{||A_0+A||}-\frac{3}{||2A_0+A||}+\frac{3}{||3A_0+A||}-\frac{1}{||4A_0+A||}\right),
\end{split}
\end{align*}
at which point the application of Lemma \ref{alg lemma} establishes the desired result. We now turn to the general case.

\noindent
\emph{Proof of Lemma \ref{lin alg craziness lemma}.\ \ }
    From Lemma \ref{cell volume lemma}, we know the sum of interest is $\frac{\sqrt n}{(n-1)!}$ times the sum taken over all integer $b$ from $b=1$ to $\infty$ of 
    \[\frac{1}{||bA_0+A_n||\;||(b+1)A_0+A_n||\;||(b+1)A_0+A_1+A_n||\cdots||(b+1)A_0+A_1+\ldots+A_n||}.\] Equivalently, let $D_0:=A_n+bA_0$ and $D_k:=A_n+bA_0+A_1+A_2+\ldots+A_{k-1}+A_k$ for 
    $k \in \{1,2,\ldots, n-1\}$. Then we are interested in 
    \begin{equation} \label{big sum in all b values lemma}
        \frac{\sqrt n}{(n-1)!}\sum_{b=1}^\infty \frac{1}{\prod_{k=0}^{n-1}||D_k||}
    \end{equation}

    Ignoring the summation and the constant in front for now, note that

    \begin{align*}
        &\frac{1}{||bA_0+A_n||\;||(b+1)A_0+A_n||\;||(b+1)A_0+A_1+A_n||\cdots||(b+1)A_0+A_1+\ldots+A_n||} \\
        >\ & \frac{1}{||bA_0+A_1+\ldots +A_n||\;||(b+1)A_0+A||\;||(b+2)A_0+A||\cdots||(b+n-1)A_0+A||},
    \end{align*}
since $b\ge 0$ and each $||A_i||\ge 0$, so we are only increasing the denominator by replacing the first line with the second.

    Using this inequality and a repeated partial fractions decomposition to produce telescoping differences, we claim that the sum in Expression \ref{big sum in all b values lemma} exceeds 
    \begin{align*}
        \begin{split}
            \frac{\sqrt n}{(n-1)!}\frac{1}{(n-1)!||A_0||^{n-1}}& \left(  \frac{{{n-2}\choose{0}}}{||A_0+A||}-\frac{{{n-2}\choose{1}}}{||2A_0+A||} \right. \\ &\left. +\frac{{{n-2}\choose{2}}}{||3A_0+A||}- \ldots + (-1)^{n-2} \frac{{{n-2}\choose{n-2}}}{||(n-1)A_0+A||}\right).
        \end{split}
    \end{align*}
    
The case for higher $n$ starts out the same way as $n=5$ but requires more steps of the partial fraction decomposition. After the final step, the number of terms we get when we cancel the tails in each set of parentheses equals the coefficient on the $||A_0||$ in the innermost denominator. This follows from the fact that 
\[\frac{1}{||rA_0+A||\;||(r+s)A_0+A||}=\frac{1}{||sA_0+A||}\left(\frac{1}{||rA_0+A||}-\frac{1}{||(r+s)A_0+A||}\right).\]

\noindent
When we take the sum over all $r$ from 1 to $\infty$, we get that the terms with coefficient $r+s$ or higher on the $A_0$ cancel, leaving exactly $s$ terms, namely $\frac{1}{||rA_0+A||}$, $\frac{1}{||(r+1)A_0+A||}, \ldots,$ and $\frac{1}{||(r+s-1)A_0+A||}$.

Clearly any time the term $\frac{1}{||rA_0+A||}$ appears for some integer $r < n-1$, we will also get the terms $\frac{1}{||(r+1)A_0+A||}$, $\ldots$, $\frac{1}{||(n-1)A_0+A||}$. This is because the innermost set of parentheses always has $\frac{1}{||b+(n-1)A_0+A||}$ subtracted, so $\frac{1}{||(n-1)A_0+A||}$ appears in every difference, possibly along with a sequence of consecutive coefficients on $A_1$ terms leading up to it.

In fact, the pattern is such that we get a second term exactly half the time, a third term a quarter of the time, and so on, with induction showing that the coefficients on the $\frac{1}{||(n-1)A_0+A||}$ terms are the multinomial coefficients for compositions appearing in their standard order (OEIS Sequence A124774) \cite{Slo}. The signs on each term correspond to how many levels of negatives appear in front, which in turn correspond to how many parts a given composition has, with an odd number getting a positive sign and an even number getting a negative sign. 

Once we have the coefficients for the $\frac{1}{||(n-1)A_0+A||}$ term in order, ascertaining the coefficients for the other terms is easy, since every other coefficient also appears with $\frac{1}{||(n-2)A_0+A||}$, every fourth with $\frac{1}{||(n-3)A_0+A||}$, every eighth with $\frac{1}{||(n-4)A_0+A||}$, etc, all the way down to $\frac{1}{||A_0+A||}$, which appears with only the first coefficient, which is always ${{n}\choose{n}}=1$.

Now the use of Lemma \ref{tree lemma} is apparent, since the coefficients are exactly the sums we have already examined in that proof.

Thus in the general case, still taking $A:=A_1+A_2+\ldots+A_n$, we have by Lemma \ref{tree lemma} that the sum over all $b$ from 1 to $\infty$ exceeds 
\begin{align*}
    \begin{split}
        \frac{\sqrt n}{(n-1)!}\frac{1}{(n-1)!||A_0||^{n-1}}\left(\frac{{{n-2}\choose{0}}}{||A_0+A||}-\frac{{{n-2}\choose{1}}}{||2A_0+A||}  +\frac{{{n-2}\choose{2}}}{||3A_0+A||} \right. \\ \left. -  \ldots + (-1)^{n-2} \frac{{{n-2}\choose{n-2}}}{||(n-1)A_0+A||}\right).
    \end{split}
\end{align*}

By Lemma \ref{alg lemma}, taking $x=||A_0||$ and $y=||A||$, we can simplify the above to \begin{align*}
    &\frac{\sqrt n}{(n-1)!}\frac{1}{(n-1)!||A_0||^{n-1}}\left(\frac{(n-2)!\cdot ||A_0||^{n-2}}{||A_0+A||\;||2A_0+A||\cdots||(n-1)A_0+A||}.\right) \\
    =& \frac{\sqrt n}{(n-1)!(n-1)||A_0||}\left(\frac{1}{||A_0+A||\;||2A_0+A||\cdots||(n-1)A_0+A||}\right).
\end{align*}

What actually matters going forward is just that this quantity, though very small, is clearly positive for any given value of $n$.
\qed

\subsection{Proof of Key Proposition \ref{prop on postive area} }

\begin{proof}
    As before, let $A:= A_1+A_2+\ldots+A_n.$ By Lemmas \ref{cell volume lemma} and \ref{lin alg craziness lemma}, we have that \begin{align*} 
    \frac{\lambda(R)}{\lambda(C)}&>\frac{\frac{\sqrt n}{(n-1)!\;(n-1)\;||A_0||}\left(\frac{1}{||A_0+A||\;||2A_0+A||\cdots||(n-1)A_0+A||}\right)}{\frac{\sqrt{n}}{(n-1)!\;||A_0||\;||A_1||\cdots||A_n||}} \\
    &=\frac{||A_1||\;||A_2||\cdots||A_n||}{(n-1)\left(||A_0+A||\;||2A_0+A||\cdots||(n-1)A_0+A||\right)}\\
    &>\frac{||A_1||^{n-1}}{(n-1)\left(||nA_n||\;||(n+1)A_n||\cdots||(2n-2)A_n||\right)} \\
    &>\frac{||A_1||^{n-1}}{(n-1)\left(||n^2A_1||\;||(n+1)nA_1||\cdots||(2n-2)nA_1||\right)} \\
    &= \frac{1}{(n-1)\left(n^2)[(n+1)n][(n+2)n]\cdots[(2n-2)n]\right)} \\
    &=\frac{1}{(n-1)(n)(n+1)(n+2)\cdots(2n-2)n^{n-1}},
    \end{align*}
    
\noindent
    where we have applied Lemmas \ref{vertices grow} and \ref{bounds on vertices} in the third and fourth lines. Since for any starting $n$, this value is a small positive constant, we know that $R$ occupies an appreciable portion of the cell $C$.
\end{proof}

\section{Algebraticity via Periodicity}\label{algebratiticity}
As mentioned in the introduction, one of the prime motivating questions for the development of multi-dimensional continued fraction algorithms is to find the correct generalization of the fact that a number is quadratic irrational if and only if its traditional continued fraction algorithm is eventually periodic.  This is the Hermite problem.

Thus the natural question  to ask is if the periodicity of the triangle map will determine something about the algebraic properties of the point $(\alpha_1, \ldots, \alpha_n)$.

\begin{thm} Let $$(\alpha_1, \ldots, \alpha_n) \in \triangle \subset \R^n$$
have an eventually periodic triangle sequence. Then $\alpha_1, \ldots, \alpha_n$ are all in the same algebraic number field whose degree is less than or equal to $n$.
    \end{thm}
    This theorem was the main goal in the $n=3$ case in \cite{Gar, Ass}.
\begin{proof}  If the triangle sequence $(b_0, b_1, b_2, \ldots)$ is eventually periodic, then certainly the sequence is bounded.  Hence the cells $\triangle(b_0, b_1, \ldots , b_m)$ will converge to a single point.  Then the arguments in \cite{Gar} for the $n=2$ case immediately apply, giving us the result.

\end{proof}

As a sketch for why this theorem is true, consider the purely periodic case.  The weak convergence, coupled with the matrices that arise in the homogeneous triangle map, lead to the vector 
$$(1, \alpha_1, \ldots, \alpha_n) \in \R^{n+1}$$
being an eigenvector of a matrix in $SL(n+1, \Z)$, which is enough to get the result.

The meat of the Hermite question is the converse.  If $\alpha_1, \ldots, \alpha_n$ are all in the same algebraic number field whose degree is less than or equal to $n$, must it be the case that the corresponding triangle sequence is eventually periodic.  We have no idea, though we suspect that this is not true.

In general, even for the cubic case, this is the difficult direction to try to prove.  There has recently been quite a bit of progress, though, in the cubic case by Karpenkov \cite{Karpenkov 2, Karpenkov 3}.

\section{Ergodicity of the Triangle Map}\label{ergodic proof}

\subsection{Basics of Ergodic Theory}

We briefly introduce the ideas from ergodic theory that will be necessary to this paper. For more complete background, see Billingsley \cite{Bil},  Dajani and Kraaikamp \cite{Daj}, Kesseb\"{o}hmer, Munday and Stratmann \cite{Kes}, and Silva \cite{Silva}.

Given a space $X$ endowed with measure $\mu$, we say a property holds \textbf{almost everywhere} if it fails only on a subset of measure 0.

A transformation $T:X\rightarrow X$ is said to be a \textbf{non-singular} transformation on the system $(X, \mathcal B, \mu)$ if for any $B \in \mathcal B$, $\mu(B)=0$ if and only if $\mu(T^{-1}(B))=0$. Note that the Farey, Gauss, Slow Triangle, and Triangle maps are all non-singular.

We say a function $T:X\rightarrow X$ \textbf{is measure-preserving}, or preserves the measure $\mu$, if for all measurable $A \subseteq X$, $\mu(A)=\mu(T^{-1}(A))$.

We say a function $T: X \rightarrow X$ is \textbf{ergodic} if whenever $T^{-1}(A)=A,$ where $A$ is a measurable subset of $X$, then either $\mu(A)=0$ or $\mu(X\setminus A)=0$. 

Given two measures $\mu$ and $\nu$ on the same space, we say $\nu$ is \textbf{absolutely continuous with respect to $\mu$} if $\mu(A)=0 \implies \nu(A)=0$, where $A \subseteq X$ . We denote this $\nu \le \le \mu$. If $\nu \le \le \mu$ and $\mu \le\le\nu$, then we say the two measures are \textbf{equivalent} and write $\mu \sim \nu$. We say equivalent measures are part of the same \emph{measure class} \cite{Kes}.

If two sets $X$ and $Y$ are equal up to a set of measure 0 using a given measure $\mu$, i.e., if $\mu(X \setminus Y)=0$ and $\mu(Y \setminus X)=0$, then we write $X = Y \pmod \mu$.

It is classical that the Gauss map on the unit interval is ergodic.  Besides the above reference, see Hensley \cite{Hensley} and Rockett and Sz\"{u}sz \cite{Rockett-Sz\"{u}sz}

\subsection{The jump transformation $g$}

{\it For the rest of this section, all the maps and domains will be in the sliced version.  Thus when we write $T$ and $\triangle$, we mean $T^S$ and $\triangle^S$, respectively.}

This subsection is a straightforward generalization of the first three paragraphs of Section 4.1 in \cite{Mes}.

 In keeping with \cite{Mes}, we introduce the jump transformation $g$. This will be useful because it allows us to jump to the points in the triangle sequence where our guaranteed cell shrinkage (i.e., our strings of zeros) occurs.  And as a reminder,  the original triangle map $T$ is ergodic if and only if the sliced triangle map $T^S$ is ergodic. 

First define 
\begin{align*}
    \begin{split}
        \mathcal B_m = \{\triangle(b_1, \ldots, b_m): \text{ the first appearance of a string of $n-1$} \\ \text{zeros in the sequence of $b_j$'s ends at $b_m$};\ i=1,2,\ldots, n\}
    \end{split}
\end{align*} and 
\[B_m = \bigcup_{E \in \mathcal B_m} E.\] 
Thus $\mathcal B_m$ is the set of all subcells of  $\triangle$ consisting of points whose triangle sequences have the first string of $n-2$ consecutive zeros being $b_{m-(n-2)}=0$, $b_{m-(n-3)}=0,$ $\ldots, b_m=0$, and $B_m$ is the cell consisting of their union. For these definitions to make sense, we need $m \ge n-2$. By Corollary \ref{b_k=0 inf often cor}, we have that 
$$\bigcup_m B_m=\triangle \pmod \lambda.$$

Note that mod $\lambda$, the sets $E$ such that $E \in \bigcup_m B_m$ form a partition $\mathcal{P}^{(1)}$ of $\triangle$. We now define 
\begin{equation*} \label{jump transformation def}
    g: \triangle \rightarrow \triangle, \quad g(u)=f^k(u)\ \ \text{ if } u \in B_k
\end{equation*}

We now look at finer and finer partitions of $\triangle,$ defined inductively as follows.  Once we have a  partition $\mathcal{P}^{(k)}$ defined, we define $\mathcal{P}^{(k+1)}$ as the partition consisting of all intersections of sets $H \in \mathcal{P}^{(k)}$ and sets $g^{-1}(G)=\bigcup_m(f^{-m}(G) \cap B_m), G \in \mathcal{P}^{(k)}.$ 

In more plain language, $\mathcal{P}^{(1)}$ is the set of all 
$$\triangle(b_0, \ldots, b_{m-n},b_{m-(n-1)},\underbrace{0,0,\ldots, 0}_{n-1 \text{ zeroes}}).$$
where all of the $b_i$ are positive, $\mathcal{P}^{(2)}$ is the set of all 
$$\triangle(b_0, \ldots, b_{m},\underbrace{0,0,\ldots, 0}_{n-1 \text{ zeroes}}),c_0, \ldots, c_{0}, \ldots, c_{k},\underbrace{0,0,\ldots, 0}_{n-1 \text{ zeroes}})),$$
where all of the $b_i$ and $c_j$ are positive, $\mathcal{P}^{(2)}$ is the set of all corresponding cones with three ``patches'' of $n-2$ zeros in a row, etc.

The key for us is that for all $\triangle(b_1, \ldots, b_m)\in \mathcal B_m, $ we have that not only is 
$$g(\triangle(b_1, \ldots, b_m)) = \triangle,$$
but that the map $g$ is one-to-one and onto.
This will mean that  for any cone $C\in \mathcal{P}^{(k)}$, 
$$g(C)= \mathcal{P}^{(k-1)}, g^{(2)}(C)= \mathcal{P}^{(k-2)}, \ldots, g^{(k-1)}(C)= \mathcal{P}^{(1)}, g^{(k)}(C)= \triangle.$$

We have that 
$$\bigcup_{C\in \mathcal{P}^k} C=\triangle \pmod \lambda.$$

As $k$ increases, we have by Theorem \ref{weakConvThm} that the diameter of all the relevant cells tends to 0. Thus $\mathcal C^{(k)} \uparrow \mathcal B$ where $\mathcal B$ is the Borel algebra and the up arrow means $\mathcal C^{(k)}$ is an increasing sequence of sets whose limit as $k\rightarrow \infty$ is $\mathcal B$.

We will be needing the following critical fact.  If $A$ is a subset with positive Lebesgue measure in $\triangle,$ then there is $k$ and a cone
$C\in \mathcal{P}^k$ so that 
$$\lambda (C\cap A) = \lambda(C),$$
meaning that $C$ is contained in $A$ almost everywhere.

To show ergodicity of $T$, all we need do is to show that the jump transformation $g$ is ergodic, by:
\begin{lem}

If $g$ is ergodic then $T$ is ergodic. 

\end{lem}

\begin{proof}

We will show that if $g$ is ergodic then $T$ is ergodic. 
    
This follows from the fact that $g^{-1}(\Omega)=\bigcup_m((T)^{-m}(\Omega) \cap B_m)$. Thus any measurable set that is its own preimage under $f$ would also be its own preimage under $g$, hence if $T$ were not ergodic, $g$ could not be either. By contrapositive, then, if $g$ is ergodic, $T$ must be as well.
\end{proof}

\subsection{The jump transformation  is ergodic}
This is a straightforward generalization of Lemma 4.1 in \cite{Mes}.

We want to show that $g:\triangle\rightarrow \triangle$
 is ergodic with respect to Lebesgue measure $\lambda.$

 Start with a measurable set $\Omega$ of $\triangle$ with 
 $$g^{-1}(\Omega) = \Omega, a.e.$$
 We want to show that 
 $$\lambda(\Omega) =0 \; \mbox{or} \; \Omega= \triangle, a.e. $$

 We will assume that $\lambda(\Omega) \neq 0,$ as otherwise we would be done.
 Suppose  $ \Omega\neq \triangle, a.e.$ Then we must have that the  complement has positive measure:
 $$\lambda(\triangle -\Omega) >0.$$
 All of our previous work in showing that the elements of the partitions 
$\mathcal{P}^{(k)}$  approach single points as $k\rightarrow \infty$ will mean that there must be a $k$ and an element of $C\in  \mathcal{P}^{(k)}$ so that $C$ is contained in $\triangle -\Omega, a.e.$ Thus we must have 
that there is a $C\in  \mathcal{P}^{(k)}$  for some $k$ with 
$$\lambda(\Omega \cap C) =0.$$

The next step is to show that for every $C\in  \mathcal{P}^{(k)}$  for any $k$, we must have 
$$\lambda(\Omega \cap C) >0,$$
giving us our needed contradiction. 

Start with some $C\in  \mathcal{P}^{(k)}$.  As $\Omega$ is an invariant set, we have 
$$\Omega \cap C= g^{(-1)}\Omega \cap C = \ldots =g^{(-k)}\Omega \cap C . $$
We know then that 
$$g^{(k)}(\Omega \cap C) = g^{(k)}\left( g^{(-k)}\Omega\right)\cap g^{(k)}(C) = \Omega. $$

Thus 
$$\lambda(\Omega \cap C) = \int_{\Omega} \mbox{Jac}(g^{(k)}) \mathrm{d} \lambda.$$
Luckily the Jacobian $\mbox{Jac}(g^{(k)})$ is not hard to calculate, as is discussed, again, in \cite{Mes}. Let 
$A=(A_0, \ldots, A_n) $ be the matrix describing the cone $C$. Then
\begin{eqnarray*}
\lambda(\Omega \cap C) &=& \int_{\Omega} \mbox{Jac}(g^{(k)}) \mathrm{d} \lambda\\
&=& \underset{\Omega }{\int} \frac{d\lambda}{(||A_0||x_0+||A_1||x_1+\ldots + ||A_n||x_n)^{n+1}} \\ 
&\geq& \frac{\lambda(\Omega)}{||A_{\mbox{max}}||^{n+1}}\\
&>&0.
    \end{eqnarray*}
where $||A_{\mbox{max}}||:=\max_{1 \le i \le n} ||A_i||$.

We have our contradiction, meaning that the map $g$ is indeed ergodic.

\section{Invariant Measures}\label{fast invariant measure}
\subsection{Invariant measures}

We now know that the triangle map
 $T:\triangle \rightarrow \triangle$, given by  

$$T(x_1, \ldots, x_n) = \left(\frac{x_2}{x_1}, \ldots , \frac{x_n}{x_1}, \frac{1-x_1-b x_n}{x_1}\right)$$
for any $(x_1, \ldots , x_n) \in \triangle_b$
is ergodic with respect to Lebesgue measure.  But this is not an invariant measure, as for almost any set $A$ of $\triangle$, we have that 
$$\lambda(A) \neq \lambda(T^{-1}(A)).$$
The natural question is to find an invariant measure that is absolutely continuous with respect to Lebesgue measure, a question both asked and answered by Gauss in the $n=1$ case, and known for $n=3$ \cite{Gar3}.
Our goal for this section is 
\begin{thm}\label{Invariant measure} The invariant measure for the triangle map $T$ in dimension $n$ is 
$$\frac{\mathrm{d} \lambda}{x_1x_2\cdots x_{n-1}(1+x_n)}.$$
\end{thm}

We will prove this using the rhetoric of the transfer operator. Recall that for a for a dynamical system with the map $T: X \rightarrow X$, the \textbf{transfer operator}, denoted $\mathcal L$ or $\mathcal L_T$,  sends measurable functions on $X$ to measurable functions on $X$ such that for all measurable subsets $A$ of $X$, we have $$\int_{T^{-1}(A)}f(x)\mathop{dx}=\int_A\mathcal L_Tf(x)\ dx.$$ 
The transfer operator is a common tool for verifying candidate invariant measures, because an invariant measure must be an eigenfunction with largest eigenvalue 1 for the transfer operator corresponding to a given map.  (For more on background of transfer operators, see Hensley \cite{Hensley},  Iosifescu and Kraaikamp \cite{I-K}, Kesseb\"ohmer, Munday and Stratmann \cite{Kes} and Schweiger \cite{Sch}, and for more general background, see Baladi \cite{Baladi00}.)

\begin{defn} The transfer operator $\mathcal L$ for any differentiable map $T$ acting on a measurable real-valued fucntion $f$  is
\[\mathcal L (f)(p) = \sum_{q:T(q)=p} \frac{1}{|\text{Jac}(T(q))|}f(q). \  \text{\cite{Gar3}}\]
\end{defn}

This section is a straightforward generalization of the $n=2$ case given in \cite{Gar3}.  Hence we simply must calculate the transfer operator for the $n$-dimensional  Triangle map and then verify that  our candidate function is an invariant measure. Since it does not require creative new ideas, we omit some of the background from dynamical systems.

We can directly calculate that 
\[T^{-1}(x_1, \ldots, x_n)=\left( \frac{1}{1+kx_{n-1}+x_n},\frac{x_1}{1+kx_{n-1}+x_n},\ldots, \frac{x_{n-1}}{1+kx_{n-1}+x_n}\right) \]
(The method is to look at the homogeneous version of the triangle map, giving us that $T^H$ can be described as an $(n+1) \times (n+1)$ matrix.  Then the inverse of $T^H$ is simply the inverse of this matrix.  To get $T^{-1}$, we simply then ``dehomogenize,'' which in this case means divide by the first coordinate, giving us the above $T^{-1}.$)

Then we get that 
\begin{align*}
    \mathcal Lf(x_1, &\ldots, x_n) \\
    =& \sum_{k=0}^\infty \frac{1}{(1+kx_{n-1}+x_n)^{n+1}}f\left(\frac{1}{1+kx_{n-1}+x_n}, \frac{x_1}{1+kx_{n-1}+x_n}, \ldots \frac{x_{n-1}}{1+kx_{n-1}+x_n}\right).
\end{align*}
(Again, while the calculations are straightforward, though a bit messy, but they are completely analogous to the $n=2$ case that is  worked out in \cite{Gar3}.)
Now for the proof of Theorem \ref{Invariant measure} .

\begin{proof}
Let $f(x_1,\ldots, x_n)=\frac{1}{x_1x_2\cdots x_{n-1}(1+x_n)}$.

Then we have 
\begin{align*}
\mathcal L f(x_1, \ldots ,x_n)&=\sum_{k=0}^\infty \frac{1}{(1+kx_{n-1}+x_n)^{n+1}}\left[\frac{(1+kx_{n-1}+x_n)^{n-1}}{x_1x_2\cdots x_{n-2}}\cdot\frac{1}{1+\frac{x_{n-1}}{1+kx_{n-1}+x_n}}\right] \\
&=\sum_{k=0}^\infty \frac{1}{(1+kx_{n-1}+x_n)^{n+1}}\left[\frac{(1+kx_{n-1}+x_n)^{n-1}}{x_1x_2\cdots x_{n-2}}\cdot\frac{1+kx_{n-1}+x_n}{1+kx_{n-1}+x_n+x_{n-1}}\right] \\
&=\frac{1}{x_1x_2\cdots x_{n-2}}\sum_{k=0}^\infty \frac{1}{1+kx_{n-1}+x_n}\left[\frac{1}{1+(k+1)x_{n-1}+x_n}\right] \\
&=\frac{1}{x_1x_2\cdots x_{n-2}}\sum_{k=0}^\infty \frac{1}{x_{n-1}}\left(\frac{1}{1+kx_{n-1}+x_n}-\frac{1}{1+(k+1)x_{n-1}+x_n}\right) \\
&=\frac{1}{x_1x_2\ldots x_{n-1}}\left(\frac{1}{1+x_n}-\frac{1}{1+x_{n-1}+x_n}+\frac{1}{1+x_{n-1}+x_n}- \ldots \right) \\
&=\frac{1}{x_1x_2\cdots x_{n-1}(1+x_n)} \\
&=f(x_1,x_2,\ldots, x_n). \\
\end{align*}
Since $\mathcal Lf(x_1,x_2,\ldots, x_n)=f(x_1, x_2, \ldots, x_n)$, we conclude that $f(x_1,\ldots,x_n)=\frac{1}{x_1x_2\cdots x_{n-1}(1+x_n)}$ is an eigenfunction for the transfer operator, hence this is the invariant measure for the $n$-dimensional triangle map on the simplex $\triangle$ with the $n+1$ vertices $(0,0,\ldots, 0, 0),$ $(1,0, \ldots,0),$ $(1,1,0, \ldots, 0),$ $(1,1,1,0, \ldots, 0), \ldots,$ and $(1,1, \ldots, 1, 1)$.
\end{proof}

\subsection{Normalizing Constant}

For the simplex $\triangle$ in $\R^n$, we have that 
$$\int_{\triangle} \frac{\mathrm{d} \lambda}{x_1x_2\cdots x_{n-1}(1+x_n)}\neq 1.$$
This subsection finds for each $n$, 
$$C_n = \int_{\triangle} \frac{\mathrm{d} \lambda}{x_1x_2\cdots x_{n-1}(1+x_n)}.$$

In the $n=1$ case (the Gauss map), we have the classical
\[C_1=\int_0^1 \frac{1}{1+x}dx =\log(2).\]

\begin{thm}
\label{zeta thm}
    For $n\ge 2$, the normalizing constant for the invariant measure of the Triangle map in $\R^{n+1}$ is 
    $$C_n=\frac{2^{n-1}-1}{2^{n-1}}\zeta(n),$$
    where $\zeta(n)$ is the Riemann zeta function.
\end{thm}
As an aside, we were mildly surprised to see the zeta function appear in these calculations.

Before we prove this, we  recall that if we define 
  \[\zeta_{even}(n)=\sum_{k=1}^\infty \frac{1}{(2k)^n},\] 
  then 
    \[\zeta_{even}(n)=\frac{1}{2^n}\zeta(n).\]
(This is a simple calculation.)

\begin{proof}
We need to calculate
\[C_n=\int_{\triangle} \frac{1}{x_1x_2\cdots x_{n-1}(1+x_n)} dx_nd x_{n-1}\ldots dx_1.\]

Parametrizing $\triangle$, this is equivalent to
\begin{equation}\label{C}
    C=\int_0^1\int_{0}^{x_1}\int_0^{x_2}\cdots \int_0^{x_{n-1}} \frac{1}{x_1x_2\cdots x_{n-1}(1+x_n)} dx_ndx_{n-1}\ldots dx_1.
\end{equation}

Using the Taylor series expansion $\frac{1}{1+x_n}=1-x_n+x_n^2-x_n^3+x_n^4-\ldots$, we have \begin{equation} \label{C as big integral}
    C=\int_0^1\int_{0}^{x_1}\int_0^{x_2}\cdots \int_0^{x_{n-1}} \frac{1}{x_1x_2\cdots x_{n-1}}\left(1-x_n+x_n^2-x_n^3+x_n^4-\ldots\right) dx_ndx_{n-1}\ldots dx_1.
\end{equation}

We proceed by induction. 
For our inductive hypothesis, we will take a close relative of the statement of the theorem, namely that 
\[\int_{0}^{x_1}\int_0^{x_2}\cdots \int_0^{x_{n-1}} \frac{1}{x_1x_2\cdots x_{n-1}(1+x_n)} dx_ndx_{n-1}\ldots dx_2=1-\frac{x_1}{2^{n-1}}+\frac{x_1^2}{3^{n-1}}-\frac{x_1^3}{4^{n-1}}+\ldots.\]

For the base case, let $n=2$. Then the integral from \eqref{C as big integral} is just 
\begin{align*}
    &\int_0^{x_1}\frac{1}{x_1}(1-x_2+x_2^2-x_2^3+x_2^4-\ldots) dx_2 \\
    =&\frac{1}{x_1}\left[x_2-\frac{x_2^2}{2}+\frac{x_2^3}{3}-\frac{x_2^4}{4}+\ldots\right]_0^{x_1} \\
    =& \frac{1}{x_1}\left(x_1-\frac{x_1^2}{2}+\frac{x_1^3}{3}-\frac{x_1^4}{4}+\ldots\right) \\
    =& 1-\frac{x_1}{2}+\frac{x_1^2}{3}-\frac{x_1^3}{4}+\ldots \\
\end{align*}

Note that this matches the statement of the inductive hypothesis for $n=2$, concluding our base case. Now suppose the statement holds for some integer $n=k\ge 2$. That means \begin{equation}
\label{IH}
    \int_0^{x_1}\int_0^{x_2}\cdots \int_0^{x_{k-1}} \frac{1}{x_1x_2\cdots x_{k-1}(1+x_k)} dx_kdx_{k-1}\ldots dx_2=1-\frac{x_1}{2^{k-1}}+\frac{x_1^2}{3^{k-1}}-\frac{x_1^3}{4^{k-1}}+\ldots.
\end{equation}

Taking the integral of both sides with respect to $x_1$ over the interval from $0$ to a new variable $x_0$ and multiplying both integrands by $\frac{1}{x_0}$ yields 
\begin{align*}
    &\int_0^{x_0}\int_0^{x_1}\int_0^{x_2}\cdots \int_0^{x_{k-1}} \frac{1}{x_0x_1\cdots x_{k-1}(1+x_k)} dx_kdx_{k-1}\ldots dx_2dx_1\\=&\frac{1}{x_0}\int_0^{x_0} 1-\frac{x_1}{2^{k-1}}+\frac{x_1^2}{3^{k-1}}-\frac{x_1^3}{4^{n-1}}+\ldots dx_1 \\=&\frac{1}{x_0}\left[x_1-\frac{x_1^2}{2^k} +\frac{x_1^3}{3^k}-\frac{x_1^4}{4^k}+\ldots\right]^{x_0}_{x_1=0} \\=&1-\frac{x_0}{2^k} +\frac{x_0^2}{3^k}-\frac{x_0^3}{4^k}+\cdots
\end{align*}

This matches our inductive hypothesis for $n=k+1$ if we relabel each $x_i$ with $x_{i+1}$. Thus the inductive hypothesis holds for all integers $\ge 2$.

Now taking the integral of both sides of Equation \eqref{IH} as $x_1$ ranges over the unit interval, we get that for any $n\ge 2$,
\begin{align*}
&\int_0^1\int_0^{x_1}\int_0^{x_2}\cdots \int_0^{x_{n-1}} \frac{1}{x_1x_2\cdots x_{n-1}(1+x_n)} dx_ndx_{n-1}\cdots dx_2dx_1 \\  
&=\int_0^1 1-\frac{x_1}{2^{n-1}}+\frac{x_1^2}{3^{n-1}}-\frac{x_1^3}{4^{n-1}}+\cdots dx_1 \\
&=\left[x_1-\frac{x_1^2}{2^n}+\frac{x_1^3}{3^n}-\frac{x_1^4}{4^n} +\cdots \right]^1_{x_1=0} \\
&= 1-\frac{1}{2^n}+\frac{1}{3^n}-\frac{1}{4^n}+\cdots \\
&=\zeta(n)-2\zeta_{even}(n) \\
&=\zeta(n)-2\frac{1}{2^n}\zeta(n) \\
&=\frac{2^{n-1}-1}{2^{n-1}}\zeta(n).
\end{align*}

By this and Equation \eqref{C}, we conclude that 
\[C_n=\frac{2^{n-1}-1}{2^{n-1}}\zeta(n),\]
completing the proof.

\end{proof}

\section{The slow triangle map}\label{slow version}
Our interest in the triangle map stems from three streams: finding algebraic numbers (the Hermite Problems section \ref{algebratiticity}), dynamical properties, and its links to partition numbers \cite{{BBDGI}, BG}.   It is the third stream that requires the use of the slow triangle map, which is why we have this section in this paper. 

\subsection{Definition of the slow triangle map}

We again start with
  $$\triangle = \{ (x_1, \ldots, x_n) \in \R^n:1\geq x_1 \cdots \geq x_n\geq 0\}.$$
We now partition this simplex into two sub-simplices:
  
 \begin{eqnarray*}\triangle_0 &=&   \{ {\bf x}\in \triangle: 1<x_1 + x_n   \} \\
 \triangle_1&=&  \{ {\bf x}\in \triangle: x_1+x_n<1 \} 
 \end{eqnarray*}
 
 \begin{defn}\label{triangle map} The slow triangle map
 $$t:\triangle \rightarrow \triangle$$ 
 is defined as 
\begin{eqnarray*}
t({\bf x}) &=& \left\{ \begin{array}{cc}  t_0({\bf x})  & \mbox{if} \; {\bf x} \in \triangle_0 \\
t_1({\bf x})  & \mbox{if} \; {\bf x} \in \triangle_1\end{array}  \right.\\
&=& \left\{ \begin{array}{cc}  \left( \frac{x_2}{x_1}, \ldots, \frac{x_n}{x_1}, \frac{1-x_1}{x_1} \right) & \mbox{if} \; {\bf x} \in \triangle_0 \\
\left( \frac{x_2}{1-x_n}, \ldots, \frac{x_n}{1-x_n}  \right)  & \mbox{if} \; {\bf x} \in \triangle_1\end{array}  \right.\
    \end{eqnarray*}

 \end{defn}
  
 When $n=1$, this is the classical Farey map.

 \begin{defn} The  slow triangle sequence  for any  ${\bf x}\in \triangle$ is a sequence 
 $$(i_0,i_1, \ldots)$$
 of  integers $i_k$, each of which is zero or one,  such that 
 $${\bf x }\in \triangle_{i_0}, t({\bf x}) \in \triangle_{i_1}, t^{(2)}({\bf x}) \in \triangle_{i_2}, \ldots .$$
 \end{defn}

 There is an easy way to pass from the slow triangle sequence for any ${\bf x}\in \triangle$ to the (fast) triangle sequence given in Definition \ref{triangle sequence}.
 For any ${\bf x}\in \triangle_b$, we have 
 $$T({ \bf x}) = t_0\circ t_1^{(b)} ({\bf x}).$$
 Thus to go from a slow triangle sequence  $(i_0,i_1, \ldots)$
to the (fast) triangle sequence we simply need to concatenate the $1$'s and add one to the number of $1$'s  Thus a slow triangle sequence 
$$(1,1,0,1,1,1,0,0,0,1,1,0, \ldots )$$
is the same as the fast triangle sequence
$$(3, 4, 1,1,3,\ldots).$$

\subsection{The invariant measure for the slow case}

As in Section \ref{fast invariant measure}, there is an invariant for the map $t$. As the arguments are similar to that earlier section, we will only sketch the results.

\begin{prop} The measure
$$\mathrm{d} \nu = \frac{1}{x_1x_2\cdots x_n}\mathrm{d} \lambda.$$
is an invariant measure for the slow triangle map $t$. Hence for all measurable sets $A$ in $\triangle$, we have 
$$\nu(A) = \int_A \mathrm{d} \nu = \int_{t^{(-1)} A} \mathrm{d} \nu = \nu(t^{(-1)} (A)).$$

\end{prop}

We need to show that the function $1/(x_1\cdots x_n)$ is an eigenfunction with eigenvalue one of the transfer operator 
\[\mathcal L f(x_1, \ldots, x_n) = |\text{det} J_0|f(T_0^{-1}(x_1,\ldots, x_n))+|\text{det}J_1|f(T_1^{-1}(x_1,\ldots, x_n)).\]

By direct calculation, we have 
\begin{eqnarray*}
\mathcal L f({\bf x}) &=& \frac{1}{(1+x_n)^{n+1}}f(T_0^{-1}({\bf x}))+\frac{1}{(1+x_n)^{n+1}}f(T_1^{-1}({\bf x})).\\
&=& \frac{1}{(1+x_n)^{n+1}}  f\left( \frac{1}{1+x_n},  \frac{x_1}{1+x_n}, \ldots , \frac{x_{n-1}}{1+x_n}\right)   \\
&& + \frac{1}{(1+x_n)^{n+1}}  f\left( \frac{x_1}{1+x_n},  \frac{x_2}{1+x_n}, \ldots , \frac{x_{n}}{1+x_n}\right)
\end{eqnarray*}

\begin{proof}
Testing $f(x_1, \ldots, x_n)=\frac{1}{x_1x_2\cdots x_n}$, we get 

\begin{align*}
    \mathcal L &f(x_1, \ldots, x_n) \\
    & = \frac{1}{(1+x_n)^{n+1}}\left[f \left( \frac{1}{1+x_n}, \frac{x_1}{1+x_n}, \ldots, \frac{x_{n-1}}{1+x_n} \right) + f \left( \frac{x_1}{1+x_n}, \frac{x_2}{1+x_n}, \ldots, \frac{x_n}{1+x_n} \right) \right]  \\
    & = \frac{1}{(1+x_n)^{n+1}}\left[ \frac{(1+x_n)^{n}}{x_1x_2\cdots x_{n-1}}
    + \frac{(1+x_n)^{n}}{x_1x_2\cdots x_n} \right] \\
    &=\frac{1}{(1+x_n)} \left[\frac{1}{x_1x_2\cdots x_n}\left(x_n+1\right) \right] \\
    &= \frac{1}{x_1x_2\cdots x_n} \\
    &= f(x_1, \ldots, x_n)
\end{align*}

\end{proof}

We cannot find the normalizing constant for this measure, as it is infinite on $\triangle_1$.  This is reflected, for example, in the $n=1$ case (the Farey map), which has an indifferent fixed point at the origin. Similarly, the origin is an indifferent fixed point for the slow triangle map in any dimension, and that the invariant measure is infinite on all of the $\triangle_1$. This should lead to many interesting technical problems and issues.

\subsection{ The induced system}
 Our goal is to show that the slow triangle map $t:\triangle \rightarrow  \triangle$ is ergodic for any dimension.  No doubt we could prove this directly, mimicking all the work that we did earlier.  But instead we will show how the ergodicity of the triangle map $T$ will imply the ergodicity of the map $t$.  We will overwhelminlgy be following the argument given in Chapter 2 of Kesseb\"{o}mer, Munday and Stratmann \cite{Kes}, where they show that the ergodicity of the Gauss map (the $n=1$ triangle map) implies the ergodicity of the Farey map (the $n=1$ slow triangle map).

By Definition 2.2.12 \cite{Kes}, the set 
$\triangle_0 =  \{ {\bf x}\in \triangle: 1<x_1 + x_n   \}$ will be a {\it sweep-out} set if 
$\nu(\triangle_0)<\infty$ and 
$$\nu(\triangle - \cup_{k=0}^{\infty} T^{(-k)} (\triangle_0)) =0.$$

\begin{lem} The set $\triangle_0$ is a sweep-out set.
  
\end{lem}

\begin{proof} It is a calculation that 
$$\nu(\triangle_0) = \int_{\triangle_0} \frac{1}{x_1x_2\cdots x_n}\mathrm{d}x_1 \cdots \mathrm{d}x_n$$
is finite.

Now, is straightforward to show that both $t_0:\triangle \rightarrow \triangle $  and $t_1:\triangle \rightarrow \triangle $  are  onto maps.  By direct calculation, we have that 
$$t_1^{(-1)} (x_1, \ldots, x_n ) = \left( \frac{x_1}{1+x_n}, \ldots, \frac{x_n}{1+x_n} \right).$$
We know that the vertices of $\triangle_0$ are 
\begin{eqnarray*}
    v_1&=& (1, 0, \ldots, 0) \\
    v_2 &=& (1, 1, 0, \ldots, 0) \\
    &\vdots& \\
    v_{n-1} &=& (1, 1, \ldots, 1, 0)
\end{eqnarray*}
    and 
\begin{eqnarray*}
    w_1&=& (1/2, 1/2, \ldots, 1/2) \\
    w_2 &=& (1, 1, \ldots, 1)
\end{eqnarray*}

We have that $t_1^{(-1)}$ acting on any of the vertices $v_i$ leaves $v_i$ alone.  But 
\begin{eqnarray*}
    t_1^{(-1)}(w_1)&=& (1/3, 1/3, \ldots, 1/3) \\
   t_1^{(-1)}(w_2) &=& (1/2, 1/2, \ldots, 1/2)
\end{eqnarray*}
In general, $t_1^{(-k)}(v_i) = v_i$ and 
\begin{eqnarray*}
    t_1^{(-k)}(w_1)&=& (1/(k+2), 1/(k+2), \ldots, 1/(k+2) \\
   t_1^{(-k)}(w_2) &=& (1/(k+1), 1/(k+1), \ldots, 1/(k+1))
\end{eqnarray*}
This shows that $\triangle_0$ is a sweep-out set.
    
\end{proof}

Definition of 2.4.25 \cite{Kes} in the context of the slow triangle map  leads to
\begin{defn} The induced map of the slow triangle map $t$  on $\triangle_0$  

$$t_{\triangle_0}: \triangle_0 \rightarrow \triangle_0 $$ is  
\[t_{\triangle_0}(x):=t^{\varphi(x)}(x).\]
where in turn 
$$\varphi(x)=\inf\{k\ge 1: n^k(x) \in \triangle_0\}.$$
\end{defn}

With the eventual goal of ergodicity for $t$, we have 
\begin{prop} If $t$ is both conservative and non-singular, then $t:\triangle \rightarrow \triangle$ is ergodic if and only if the induced map $t_{\triangle_0}:\triangle_0 \rightarrow \triangle_0$ is ergodic.
\end{prop}
This is simply Proposition 2.4.28 in \cite{Kes}.

Now, the triangle map $t$ is non-singular (Definition 2.2.4 \cite{Kes}) if for all measurable sets $A$  with measure zero we have that the measure of $t^{(-1)}(A)$ also has measure zero. But this follows from the existence of the invariant measure $\nu$. 

The map $t$ is conservative follows from $t_{\triangle_0}$ being conservative, which follows immediately from Proposition 2.2.23 \cite{Kes}, and from Proposition 2.4.27 \cite{Kes}.

\subsection{Ergodicity of induced map}

Thus we have to show that $t_{\triangle_0}:\triangle_0 \rightarrow \triangle_0$ is ergodic.

We need to link the induced system to the original triangle map $T$, which we know to be ergodic.

\begin{prop} The dynamical system
$$t_{\triangle_0}:\triangle_0 \rightarrow \triangle_0$$
with invariant measure $\nu|_{\triangle_0}$
 is measure-theoretically isomorphic to the dynamical system 
 $$T:\triangle \rightarrow \triangle$$
 with invariant measure $\mu.$

\end{prop}
This will immediately give us that the slow triangle map is ergodic.

As to measure theoretically isomophic, this means the following.  We  have to construct a map 
$$\psi:\triangle_0 \rightarrow \triangle$$
so that the diagram

$$\begin{array}{ccc}
\triangle_0 & \xrightarrow{t_{\triangle_0}} & \triangle_0 \\
\psi \downarrow & & \downarrow \psi \\
\triangle & \xrightarrow{T} & \triangle 
\end{array}$$
is commutative and so that, for all measurable sets $B$ in $\triangle,$ we have 
$$\nu|_{\triangle_0}\circ \psi^{-1}(B) = \mu(B).$$

\begin{proof} Simply define 
$$\psi = t_0.$$
Then the argument is the same as that given in the proof of Proposition 2.4.31 in \cite{Kes}, which is for the $n=1$ case. 
\end{proof}
---------------------------------------------
----------------------------------------------
\section{On our generalization of the triangle map}\label{generalization}

The triangle map in Section \ref{Non-homgeneous Triangle Map} agrees with the Gauss map for $n=1$ and with the map in \cite{Gar} for $n=2$ (and also up to conjugacy with the map in \cite{Mes}.  It is not the generalization in higher dimensions that was proposed in \cite{Gar} and briefly discussed in \cite{Mes}.  But that earlier generalization was not the ``right'' one, for at least two reasons.

First, the triangle map in this paper has a natural interplay with a slow version, as seen in Section \ref{slow version}.  As a good multi-dimensional continued fraction algorithm should be a generalized continued fraction algorithm, there should be such a natural interplay between fast and slow versions.

Second, the slow version in Section \ref{slow version} is precisely the version that seems to be particularly compatible with the study of partition numbers \cite{BBDGI, BG}. Much of our motivation for trying to understand the dynamics of this triangle in higher dimensions is in an attempt to link its dynamics with partition theory, work that is ongoing.

\section{Conclusion and Questions}\label{Conclusions}
We have shown that the triangle map (both the fast and the slow version) is ergodic in every dimension.  Key is in showing that we have weak convergence almost everywhere, as shown in Section \ref{almost everywhere weak convergence}. 

There are of course many questions left.  For example, once we have ergodicity, what other stronger  dynamical properties does the triangle map have, such as mixing.  

In Section \ref{fast invariant measure} we found the  transfer operator for the triangle map in each dimension.  There is a rich tradition and many wonderful papers studying the spectrum the transfer operator for the Gauss map \cite{Hensley, I-K, Kes, Sch, Baladi00}.  In \cite{Gar3}, the spectrum of the transfer operator for the $n=2$ triangle map was studied.  More needs to be done, as little is known about the spectrum of the higher dimensional triangle maps.

As mentioned in the introduction, a large part of our motivation for this paper is the link between the triangle map and integer partitions.  It appears that the dynamics of the triangle map gives us information about integer partitons, and symmetries in integer partitions tell us information about the triangle map.  This is being explored in \cite{Fox-Garrity}.

We have shown that the cells $C(b_1, b_2, \ldots)$ converge to points almost everywhere.  But we need the phrase ``almost everywhere.''  While in the classical case of the Gauss map ($n=1$), every cell does converge to an isolated point, this is not the case for $n=2$.  In \cite{Ass}, an explicit description of the growth rate of the $b_1, b_2, \ldots$ is given for when the cells converge not to a point but to a line segment.  What happens in higher dimensions?  We strongly conjecture that in the dimension $n$ case, the cells could converge to any $k$-simplex, for $0\leq \leq n-1.$ Further, what are the conditions needed on the triangle sequence  $b_1, b_2, \ldots$ to  have the corresponding cell converge to a $k$-simplex?

\end{document}